\documentclass[journal,twoside,web]{ieeecolor}

\usepackage{ulem}
\usepackage{soul}
\usepackage[dvipsnames]{xcolor}
\usepackage{generic}
\usepackage{cite}
\usepackage{amsmath,amssymb,amsfonts,extarrows}
\usepackage{algorithmic}
\usepackage{graphicx}
\usepackage{textcomp}
\usepackage{bm}
\usepackage{tikz}
\usepackage{color}
\usepackage{setspace}
\usepackage{subfigure}
\usetikzlibrary{positioning,shapes,arrows}
\def\BibTeX{{\rm B\kern-.05em{\sc i\kern-.025em b}\kern-.08em
    T\kern-.1667em\lower.7ex\hbox{E}\kern-.125emX}}

\usepackage[switch]{lineno}

\newtheorem{theorem}{\bf Theorem}{}
\newtheorem{corollary}{\bf Corollary}{}
\newtheorem{remark}{\bf Remark}{}
\newtheorem{lemma}{\bf Lemma}{}
\newtheorem{proposition}{\bf Proposition}
\newtheorem{definition}{\bf Definition}{}
\newtheorem{assumption}{\bf Assumption}{}
\newtheorem{problem}{\bf Problem}[subsection]
\newtheorem{example}{\bf Example}{}

\newcommand{\NNsimple}{
    \tikzstyle{inputNode}=[circle,draw=magenta,fill=blue!10,minimum size=6pt,inner sep=0pt]
    \tikzstyle{stateTransition}=[-stealth, thin]
    \begin{tikzpicture}
                
        \node[inputNode, thick] (i1) at (0, 0.2) {};
        \node[inputNode, thick] (i2) at (0, -0.2) {};

        \node[inputNode, thick] (h11) at (0.75, 0.35) {};
        \node[inputNode, thick] (h12) at (0.75, 0.0) {};
        \node[inputNode, thick] (h13) at (0.75, -0.35) {};

        \node[inputNode, thick] (h21) at (1.5, 0.35) {};
        \node[inputNode, thick] (h22) at (1.5, 0.0) {};
        \node[inputNode, thick] (h23) at (1.5, -0.35) {};

        \node[inputNode, thick] (h31) at (2.25, 0.2) {};
        \node[inputNode, thick] (h32) at (2.25, -0.2) {};

        \draw[stateTransition] (i1) -- (h11);
        \draw[stateTransition] (i1) -- (h12);
        \draw[stateTransition] (i1) -- (h13);
        \draw[stateTransition] (i2) -- (h11);
        \draw[stateTransition] (i2) -- (h12);
        \draw[stateTransition] (i2) -- (h13);

        \draw[stateTransition] (h11) -- (h21);
        \draw[stateTransition] (h11) -- (h22);
        \draw[stateTransition] (h11) -- (h23);
        \draw[stateTransition] (h12) -- (h21);
        \draw[stateTransition] (h12) -- (h22);
        \draw[stateTransition] (h12) -- (h23);  
        \draw[stateTransition] (h13) -- (h21);
        \draw[stateTransition] (h13) -- (h22);
        \draw[stateTransition] (h13) -- (h23);

        \draw[stateTransition] (h21) -- (h31);
        \draw[stateTransition] (h21) -- (h32);
        \draw[stateTransition] (h22) -- (h31);
        \draw[stateTransition] (h22) -- (h32);
        \draw[stateTransition] (h23) -- (h31);
        \draw[stateTransition] (h23) -- (h32);

        \draw[stateTransition] (-0.375, 0.2) --  (i1);
        \draw[stateTransition] (-0.375, -0.2) -- (i2);
        \draw[stateTransition] (h31) -- (2.625, 0.2);
        \draw[stateTransition] (h32) -- (2.625, -0.2);
    \end{tikzpicture}
}

\newcommand*{\num}{4}

\definecolor{rose_red}{RGB}{230,28,93}
\definecolor{heavy_purple}{RGB}{147,0,119}
\definecolor{heavy_blue}{RGB}{58,0,136}

\begin{document}

\title{Neural Observer with Lyapunov Stability Guarantee for Uncertain Nonlinear Systems}

\author{Song Chen, Shengze Cai, Tehuan Chen, Chao Xu, and Jian Chu
\thanks{Manuscript received xx xx, 202x; accepted xx xx, 202x. This work was supported by the National Key Research and Development Program of China under Grant 2019YFB1705800, the National Natural Science Foundation of China number 61973270, the Zhejiang Provincial Natural Science Foundation of China number LY21F030003, and the Zhejiang Provincial Natural Science Foundation of China number LY19A010024. (\textit{Corresponding author: Chao Xu}.)}
\thanks{S. Chen is with the School of Mathematical Sciences, Zhejiang University, Hangzhou, Zhejiang 310027, China (e-mail: math\_cs@zju.edu.cn).}
\thanks{T. Chen is with School of Mechanical Engineering and Mechanics, Ningbo University, Ningbo, Zhejiang 315211, China, and also with Ningbo Artificial Intelligence Institute, Shanghai Jiao Tong University, Ningbo, Zhejiang 315000, China.  (e-mail: chentehuan@nbu.edu.cn).}
\thanks{S. Cai and J. Chu are with the State Key Laboratory of Industrial Control Technology, Institute of Cyber-Systems and Control, Zhejiang University, Hangzhou, Zhejiang 310027, China (e-mail: shengze\_cai@zju.edu.cn, chuj@iipc.zju.edu.cn).}
\thanks{C. Xu is with the State Key Laboratory of Industrial Control Technology, Institute of Cyber-Systems and Control, Zhejiang University, Hangzhou, Zhejiang 310027, China, and also with Huzhou Institute of Zhejiang University, Huzhou, Zhejiang 313000, China.  (e-mail: cxu@zju.edu.cn).}
}

\maketitle

\begin{abstract}
In this paper, we propose a novel nonlinear observer based on neural networks, called neural observer, for observation tasks of linear time-invariant (LTI) systems and uncertain nonlinear systems. 
In particular, the neural observer designed for uncertain systems is inspired by the active disturbance rejection control, which can measure the uncertainty in real-time. 
The stability analysis (e.g., exponential convergence rate) of LTI and uncertain nonlinear systems (involving neural observers) are presented and guaranteed, where it is shown that the observation problems can be solved only using the linear matrix inequalities (LMIs). 
Also, it is revealed that the observability and controllability of the system matrices are required to demonstrate the existence of solutions of LMIs.
Finally, the effectiveness of neural observers is verified on three simulation cases, including the X-29A aircraft model, the nonlinear pendulum, and the four-wheel steering vehicle.
\end{abstract}


\section{Introduction}
    \IEEEPARstart{W}{ith} the success of machine learning (ML) algorithms in various complex tasks such as computer vision and natural language processing, connecting ML with control theory  has become a hot topic in recent years and is attracting more and more researchers \cite{jordan2015machine,tsukamoto2021contraction,lukas2022safe}. On the one hand, the data-driven ML methods have been widely used to deal with nonlinear control problems, which can be traced back to early years when the neural network (NN) theory was proposed \cite{werbos1989neural,levin1993control,levin1996control,ge2013stable}. 
    However, it is not easy to utilize the model information to construct the control input when the model itself contains uncertainties (e.g., unmodeled dynamics). 
    To tackle this challenge, more effective modeling methods based on deep NNs are gradually coming into our vision recently, such as physical-informed ML \cite{karniadakis2021physics}, stable deep dynamics learning \cite{kolter2019learning}, and neural operator learning \cite{lu2021learning}. There are also many works proposed to address high-dimensional control problems (e.g., solving the Hamilton-Jacobi-Bellman equation) and state observation problems via learning methods~\cite{han2018solving,bottcher2022ai,breiten2021neural,chakrabarty2021safe,abdollahi2006stable,nguyen2021neural,qifeng2021gaussian}, indicating that the ML methods can be successfully employed in control and identification problems. 
    On the other hand, the classical control theories are conversely applied to explain why and how the ML algorithms work~\cite{li2018maximum,lessard2016analysis,lin2021control,el2021implicit}. For example, the convergence performance of optimization algorithms can be analyzed via linear matrix inequalities (LMIs) \cite{lessard2016analysis}.
    

    \textcolor{black}{Despite the aforementioned advances, there are still some intractable challenges in learning-based control via deep NNs. For examples, how to directly analyze the control performance (e.g., stability, optimality, etc.) of a system equipped with NN mappings remains a problem~\cite{tsukamoto2021contraction,lukas2022safe}. It is not straightforward to apply the nonlinear control theory \cite{fazlyab2022safety,tipaldi2022reinforcement}, as there are various types of nonlinear activation functions and numerous parameters in an NN mapping. {\color{black}Additionally, such systems are generally vulnerable to various malicious perturbations \cite{chen2020towards} due to the black-box nature of deep NNs.} Furthermore, the training process is highly dependent to the data, thus one needs to appropriately select the sampling method for system state and consider the training data distribution, reducing the impact on the closed-loop system~\cite{lederer2020training,lederer2021impact}. Lastly, how to interpret that the trained NNs are applicable is also an open question in the ML community.}
    
    {\color{black} In this paper, we mainly focus on the state observation tasks, where the observers are designed based on neural networks. Following \cite{chakrabarty2021safe,abdollahi2006stable,nguyen2021neural}, the following dynamical model and the corresponding observer are considered:
    \begin{equation*}
    \begin{aligned}
        &\text{Dynamical model:}
        \left\{
        \begin{aligned}
            \dot{x}&=Ax+g(x,u)\\
            y&=Cx,
        \end{aligned}
        \right.\\
        &\text{NN-based Observer:}
        \left\{
        \begin{aligned}
            \dot{\widehat{x}}&=A\widehat{x}+\pi_{\theta}(\widehat{x},u)+G(y-C\widehat{x})\\
            \widehat{y}&=C\widehat{x},
        \end{aligned}
        \right.
    \end{aligned}
    \end{equation*}
    where $(C, A)$ is observable and the NN $\pi_{\theta}(x,u)$ in the observer is trained to approximate the uncertainty $g(x,u)$. Then, one could provide an NN-based observer to achieve  $\widehat{x}\rightarrow x$. In this context, we are motivated to ask a question: \textit{how can we find a concise condition to verify the availability of {\color{black}an NN} for the system, with which the performance is not limited by the sampling method and can be directly analyzed?} If the condition exists, most of the aforementioned challenges can be addressed. } 
    
    Recently,  \cite{yin2022stability} and~\cite{pauli2021offset} proposed an efficient method, based on  quadratic constraints (QC) and linear matrix inequality (LMI), {\color{black}to analyze the robust stability in equilibrium points of systems controlled by one state-feedback NN mapping controller $u(k)=\pi_{\theta}(x(k))$.} {\color{black} However, the analysis of robust stability in \cite[Theorem. 2]{yin2022stability} highly depends on the assumption that the perturbation $\Delta$ is bounded and depends on the skilled construction of filter $\Psi_{\Delta}$, which is applied to capture the correlation between the input and output signals of $\Delta$ against time. } 
    As for the LMIs conditions that guarantee the stability in \cite{yin2022stability,pauli2021offset}, they do not explicitly indicate whether the solutions of LMIs exist or not. Moreover, the filter $\Psi_{\Delta}$ may also complicate to solve the LMI condition \cite[Theorem. 2]{yin2022stability}. 
    In our work, {\color{black}instead of constructing a filter $\Psi_{\Delta}$ for the uncertain systems}, we design neural observers inspired by the essential philosophy of active disturbance rejection control (ADRC) proposed in \cite{han2009adrc}, where the basic idea is to regard the ``total uncertainty'' as an extended state of the system. By applying ADRC, one can estimate the uncertainty and compensate it in the control input in real-time. The theorems about ADRC can be found in \cite{guo2016active,freidovich2008performance}.

\subsection{Paper Contribution}
    The contributions in this paper can be summarized as follows:
    
    $\bm{(1)}$ This work belongs to the category of using machine learning in control problems. {\color{black} We introduce a specially structured NN mapping $\pi_{\theta}(\cdot)$ to design nonlinear neural observers for the observation task of dynamical systems.} We first propose two relative definitions: \textit{neural observable} and {\color{black}\textit{neural exponentially observable}}. Ideologically, for a controllable and observable linear time-invariant (LTI) system, we construct a Luenberger-form neural observer derived from the feedback of errors and employ the {\color{black}estimated state} $\widehat{x}$ to design the feedback NN control law $u=\pi_{\theta}(\widehat{x})$. For two classes of nonlinear systems (i.e., \textcolor{black}{integrator chain} nonlinear systems and MIMO nonlinear systems consisting a linear dynamic part and the uncertainty), we respectively design the corresponding neural observers to measure the state and the ``total uncertainty'' by inheriting the idea of ADRC \cite{han2009adrc}. More details are given in Sections \ref{sec: Neural Observer} and \ref{sec: problem formulation}.
    
    $\bm{(2)}$ {\color{black} We develop the NN isolation method and QCs (see Lemma \ref{lem: QC for K mappings}) for {\color{black}an NN} mapping vector $\bm{\pi_{\theta}}(\cdot)$, which is composed of $K$ NN mappings, i.e., $\bm{\pi_{\theta}}(\cdot)=[\pi^{\top}_{\theta_1}(\cdot),\cdots,\pi^{\top}_{\theta_{K}}(\cdot)]^{\top}$ with parameters $\bm{\theta}=(\theta_1,\cdots,\theta_K)$.} In addition, we point out that Lemma \ref{lem: QC for K mappings} can be used to deduce the linear matrix inequality (LMI) of closed-loop dynamics under a feedback interconnection. More details can be found in Section \ref{sec: NN Representation and NN mapping vector}.
    
    $\bm{(3)}$ {\color{black} We provide a verification framework for the availability of NNs in NN-based systems via LMIs.} The first and the second results (see Theorem \ref{theo: LMI for LTI without control}-\ref{theo: LMI for LTI with control}) provide LMI conditions to guarantee the neural exponential observability and globally exponential stability for LTI systems, respectively. Furthermore, in these cases, we reveal the relationship between the existence of solutions for LMIs and the observability and controllability of LTI systems (see Proposition \ref{prop: existence of solution of LMI}-\ref{prop: existence of solution of LMI 2}). Under this fundamental framework, we provide the third and the fourth results (see Theorem \ref{theo: LMI for integral-chain systems}-\ref{theo: LMI for nonintegral-chain systems}), which achieve the neural observability for \textcolor{black}{integrator chain} nonlinear systems and a class of MIMO nonlinear systems, respectively. Different from Theorem \ref{theo: LMI for LTI without control}-\ref{theo: LMI for LTI with control}, Theorem \ref{theo: LMI for integral-chain systems}-\ref{theo: LMI for nonintegral-chain systems} can not only guarantee the observability but also measure the uncertainty in real-time.
    
    This paper is organized as follows. In Section \ref{sec: Neural Observer}, we present the key ideas of the neural observer. In Section \ref{sec: problem formulation}, we propose two definitions of observability as well as the formulation of neural observers for different kinds of systems. Moreover, relevant observation problems are also defined. 
    Section \ref{sec: NN Representation and NN mapping vector} discusses the NN isolation and QCs method for the NN mapping and the NN mapping vector. Then, the convergence analysis associated with the observation problems are provided in Section \ref{sec: Main Result}. Finally, in Section \ref{sec: Numerical Experiments}, we provide the simulation results to verify the efficiency of our framework. 

    \subsection{Mathematical Notation} 
    
    $\bullet$ $\mathbf{R}^n$ denotes $n$-dimensional real linear space. In this paragraph, only the real linear and finite dimensional spaces are considered. Each space, $\mathcal{M}$, holds an inner product $\langle w, v\rangle=w^{\top}v$ and a norm $\|w\|_2^2=\langle w, w\rangle$. 
    We use pointwise orders $\geq,>$ for any vectors $w,v \in \mathcal{M}$, i.e., $w \geq (>) v \Longleftrightarrow w_i\geq (>) v_i$, $i\in\{1,\cdots,\dim(\mathcal{M})\}$. The set of real numbers in the interval $[a, b] \subset \mathbf{R}$ is denoted by $\mathcal{T}_{[a,b]}$, and the set of real numbers in the interval $[a, \infty) \subset \mathbf{R}$ is $\mathcal{T}_{\geq a}$. 
    $\mathcal{L}(\mathcal{M}, \mathcal{N})$ accounts for the space of all linear mappings (matrices) from ``$\mathcal{M}$'' to ``$\mathcal{N}$''. For any mappings $T \in \mathcal{L}(\mathcal{M}, \mathcal{N})$, the induced $2$-norm is defined by $\|T\|_2=\sup_{x\in \mathcal{M}, \|x\|_2=1}\|Tx\|_2$. Specially, when the linear spaces $\mathcal{M}$ and $\mathcal{N}$ are identical, $\mathcal{L}(\mathcal{M}, \mathcal{N})$ can be abbreviated to $\mathcal{L}(\mathcal{M})$. 
    Given a mapping $T \in \mathcal{L}(\mathcal{M})$, $T\succ 0$ represents $T=T^{\top}$ and $\langle Tw, w\rangle \geq 0$ for all $w \in \mathcal{M}$, where ``$=$" is true if and only if $w=0$; and $\lambda_{\max}(T)$, $\lambda_{\min}(T)$ denote the maximum and minimum eigenvalue, respectively. 
    In addition, $\operatorname{diag}(A_1,\cdots,A_n)$ represents a diagonal block matrix, where the $i_{\text{th}}$ diagonal block is $A_i$. $\bm{0}_n,\bm{1}_m$ are $n$-dimensional zero vector and $m$-dimensional vector whose entries are all ones, respectively.
    
    $\bullet$ The space of $k$-th continuously differentiable functions from $\mathcal{M}$ to $\mathcal{N}$ is denoted by $C^k(\mathcal{M}; \mathcal{N})$. Similarly, when the spaces $\mathcal{M}$ and $\mathcal{N}$ are identical, the notation $C^k(\mathcal{M})$ is used for short. For any differentiable vector function $\mathcal{F}(\bm{x})$, $\mathcal{F}_{x_i}$ represents the partial derivative of $\mathcal{F}$ with respect to $x_i$.  $\mathcal{O}(\alpha^k)$ is said to be the infinitesimal of $k$-order of $\alpha$ if $\lim_{\alpha \rightarrow 0}\frac{\mathcal{O}(\alpha^k)}{\alpha^k} = c$, $c\neq 0$. 

\section{Neural Observer}
\label{sec: Neural Observer}
    \tikzstyle{G}=[draw,rounded corners,fill=black!20!red!20!white, minimum height=3em, minimum width=7em,very thick]
    \tikzstyle{N}=[draw,rounded corners,fill=black!20!blue!20!white, minimum height=3em, minimum width=7em,very thick]
    \tikzstyle{sum} = [draw,circle, fill=black!20!blue!40!white,radius=1em,very thick]
    
    In this paper, we focus on an underlying, but not necessarily single-input single-output (SISO) or multiple-input multiple-output (MIMO), continuous-time nonlinear and uncertain system formulated by an ordinary differential equation (ODE):
    \begin{equation}
        \begin{aligned}
            \dot{x}(t)&=f(x(t),u(t),w(t),t),\\
            y(t) &= Cx(t),
        \end{aligned}
        \label{eq: general systems}
    \end{equation}
    where $x(t)\in \mathbf{R}^{n_s}$, $u(t) \in \mathbf{R}^{n_u}$, and $y(t) \in \mathbf{R}^{n_o}$ denote the state, the control input, and the control output, respectively. Generally, the external disturbances $w \in \mathbf{R}^{n_d}$ satisfying $\sup_{t \in \mathcal{T}_{\geq 0}}\|(w,\dot{w})\|<\infty$ are considered in dynamical systems. $f \in C^1(\mathbf{R}^{n_s} \times \mathbf{R}^{n_u} \times \mathbf{R}^{n_d}; \mathbf{R}^{n_s})$ {\color{black}is a nonlinear function} called the \textit{total uncertainty}, might be partially unknown or totally unknown. $C \in \mathcal{L}(\mathbf{R}^{n_s},\mathbf{R}^{n_o})$ denotes the observation matrix of the system.
    
    Due to uncertainty and disturbance, the direct measurement of state would be costly and less credible. Nevertheless, the output measurement is convenient to obtain. 
    Hence, {\color{black}\textbf{our observation objective} is to design an NN based output-feedback observer (neural observer)}, such that the state of system \eqref{eq: general systems} is \textit{globally observable} for any initial state $x(0)$ and total uncertainty. 
    
    To present the structure of the neural observer, we first introduce an output-feedback neural network (NN) mapping $\pi_{\theta}(\cdot)$ with a parameter $\theta$. We consider $\pi_{\theta}(\cdot)$ as a feed-forward NN with $L$ hidden layers and activation functions $\sigma(\cdot) \in C(\mathbf{R}|\sigma(0)=0)$ that are identical in all layers. It should be noted that the input is the $0_{\text{th}}$-layer and the output is given by the $(L+1)_{\text{th}}$-layer, i.e., $\pi_{\theta}^{[0]}(x)=x$ and $\pi_{\theta}^{[L+1]}(x)=\pi_{\theta}(x)$. Let $n_l$ be the number of neurons in $l_{\text{th}}$-layer. By given weights matrix $W^{l} \in \mathcal{L}(\mathbf{R}^{n_{l-1}},\mathbf{R}^{n_l}) $, the $(L+4)$-tuple parameter $\theta$ and the NN mapping $\pi_{\theta}(\cdot)$ are defined as follows:
    \begin{equation}
        \begin{aligned}
            \theta &=\left(L,n_{\sigma},W^{1},\cdots,W^{L+2}\right),\ n_{\sigma}\triangleq \sum_{i=1}^L n_i\\ 
            \pi^{[l]}_{\theta}(x)&=\sigma^{[l]}\left(W^{l}\pi^{[l-1]}_{\theta}(x)\right),\ l=1,\cdots,L,\\
            \pi_{\theta}(x)&=W^{L+1}\pi^{[L]}_{\theta}(x)+W^{L+2}\pi_{\theta}^{[0]}(x),
        \end{aligned}
        \label{eq: resnets}
    \end{equation}
    where $\sigma^{[l]}(x)=[\sigma(x_1),\cdots,\sigma(x_{n_l})]^{\top}$. We note that when $W^{L+2}\neq O$, the NN mapping is called the residual neural network proposed by He et al. \cite{he2016deep} and shown in Fig. \ref{fig: resnets}.
    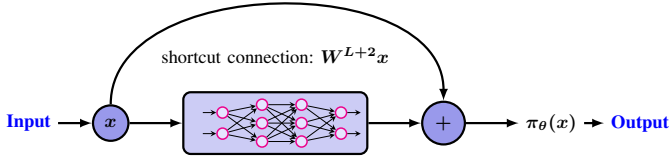
\begin{figure}[t]
        \centering
        \scalebox{0.7}{
        \begin{tikzpicture}[auto, node distance=2cm]
            \node [N] (observer) {\NNsimple};
            \node[sum, inner sep=4pt, left=1cm of observer.west, very thick] (input) {$\bm{x}$};
            \node[sum, inner sep=4pt, right=1cm of observer.east,very thick] (sum) {$\bm{+}$};
            \node[ right=1cm of sum.east,very thick] (output) {$\bm{\pi_{\theta}(x)}$};
            
            \draw[latex-,very thick] (input) --++ (-1cm,0)node[left,pos=1]{\textcolor{blue}{\textbf{Input}}};
            \draw[-latex,very thick] (input)--(observer);
            \draw[-latex,very thick] (observer)--(sum);
            \draw[-latex,very thick] (sum) --(output); 
            \draw[-latex,very thick] (input.north) to [out=90,in=90](sum.north) node[yshift=1cm,pos=0.5]{shortcut connection: $\bm{W^{L+2}x}$};
             \draw[-latex,very thick] (output) --++ (1cm,0)node[right,pos=1]{\textcolor{blue}{\textbf{Output}}};
        \end{tikzpicture}
        }
        \caption{Residual neural network: a feed-forward NN with $L$ hidden layers and a shortcut connection.}
        \label{fig: resnets}
    \end{figure}
    
    Our work focuses on such {\color{black}an NN} mapping, which facilitates revealing the existence of the parameter $\theta$ in observer $\pi_{\theta}(\cdot)$ that would be demonstrated in Remark \ref{rm: role explanation for shortcut}. Based on the defined NN, we construct the neural observer as following:
    \begin{equation}
    \left\{
        \begin{aligned}                \dot{\widehat{x}}(t)&=g(\widehat{x}(t),u(t),\bm{\pi_{\theta}}(y-\widehat{y})),\\
            \widehat{y}&=C\widehat{x},
        \end{aligned}
    \right.
    \label{eq: general nn observers}
    \end{equation}
    where $\widehat{x} \in \mathbf{R}^{n_{\widehat{s}}}$ and $\widehat{y} \in \mathbf{R}^{n_o}$ are {\color{black}estimated state and estimated output}, respectively. $\bm{\pi_{\theta}}(y-\widehat{y})=[\pi^{\top}_{\theta_1}(y-\widehat{y}),\cdots,\pi^{\top}_{\theta_{K}}(y-\widehat{y})]^{\top}$ represents the NN mapping vector, where $\theta_i=(L_i,n_{\sigma_i}, W_i^{1},\cdots,W_i^{L_i+2})$. Moreover, $g$ is a known and continuous function.
    The block diagram of the neural observation framework is shown in Fig. \ref{fig: neural observation framework}.
    
    \begin{figure}[t]
        \centering
        \scalebox{.7}{
            \begin{tikzpicture}[auto]
                \node[G](plant) at (0em,0em) {Plant};
                \node[N, below=4em of plant] (NN) {Neural Observer};
                \draw[very thick,latex-](plant.west)--++(-4em,0)node[above,pos=0.5,name=u]{$u$};
                \draw[very thick,-latex](plant.east)--++(4em,0)node[above,pos=0.5,name=y]{$y$};
                \draw[very thick,-latex](u)|-(NN.west);
                \draw[very thick,-latex](NN.east)--++(4em,0)node[above,pos=0.45,name=hat_y]{$\widehat{y}$};
                \node[sum,below=2em of y](sum){$\bm{+}$};
                \draw[very thick,-latex](y)--(sum.north);
                \draw[very thick,-latex](hat_y)--(sum.south)node[right,pos=0.8]{$\bm{-}$};
                \draw[very thick,-latex](sum.west)-|(NN.north);
            \end{tikzpicture}
        }
        \caption{The block diagram of the neural observation framework: the observed error $y(t)-\widehat{y}(t)$ is the input of NN mapping vector in the neural observer.}
        \label{fig: neural observation framework}
    \end{figure}
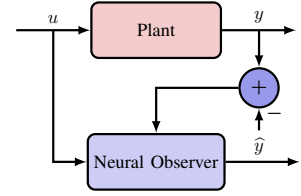
    
    \begin{remark}[The keys to neural observers]
    \quad\quad\quad\quad\quad\quad
    \begin{itemize}
        \item[(1)] It is worth noting that the dimension of $\widehat{x}$ is designed to be no lower than the state $x$, i.e., $n_{\widehat{s}} \geq n_s$, since the relatively higher dimensional information may improve the estimation performance, which is similar to the well-known kernel trick in machine learning algorithms (such as Gaussian process regression \cite{rasmussen2005gaussian}).
        \item[(2)] Furthermore, \textcolor{black}{the construction of the continuous function $g$ follows the ``white box modeling'' information, including but not limited to the known system matrices $A,B,C$ in the LTI system (see 
        details in Sec. \ref{sec: problem formulation} A), the integrator chain structure and the order of models $n$ in the integrator chain system (see Sec. \ref{sec: problem formulation} B), the known system matrices $A, B, C, B_w$ in the nonlinear system (see Sec. \ref{sec: problem formulation} C), and so on. }
        \item[(3)] \textcolor{black}{Moreover, in order to reduce the computational complexity, the parameters of different NNs in an NN mapping vector $\bm{\pi_{\theta}}(\cdot)=[\pi^{\top}_{\theta_1}(\cdot),\cdots,\pi^{\top}_{\theta_{K}}(\cdot)]^{\top}$ could be identical. For example, $\theta_{n_i}=\theta_{n_j}$ for $i,j \in I$, where $I$ is an index-subset of $\{1,\cdots,K\}$.}
    \end{itemize}
    \label{rm: keys for NO}
    \end{remark}

\section{Problem Formulation}
\label{sec: problem formulation}
    

    Before presenting the main results, we first clarify the main problems that we will consider in this paper. We aim to  find a family of architectures, including $\bm{\pi_{\theta}}$ and $g(\widehat{x},u,\bm{\pi_{\theta}})$, for output-feedback observation tasks. Moreover, the architectures are expected to ensure the existence of NN parameter $\bm{\theta}$ and enable the implementation of neural observers. 
    To analyze the neural observers theoretically, we propose the following definitions:
    
    \begin{definition}[\textbf{neural observable}]
    
   We suppose that there exists the NN mapping vector $\bm{\pi_{\theta}}$ such that the closed-loop system composed of \eqref{eq: general systems} and \eqref{eq: general nn observers} satisfies for all $x(0),\widehat{x}(0)$,
    $$
    \|x(t)-T\widehat{x}(t)\|_2\rightarrow0, \text{in some sense, for example,}\ t\rightarrow +\infty,
    $$
    where $T\in\mathcal{L}(\mathbf{R}^{n_{\widehat{s}}},\mathbf{R}^{n_s})$ with $T_{ij}=1,\ i=j;\ T_{ij}=0,\ {\color{black}\text{else}}$.
    Then we say that the system \eqref{eq: general systems} is \textbf{neural observable}.
    \label{def: neural observable}
    \end{definition}
    
   We note that $T=I$ if $n_s$ is equal to $n_{\widehat{s}}$. And $T$ is also applied in the following definition.
    \begin{definition}[{\color{black}\textbf{neural exponentially observable}}]
   We consider the aforementioned closed-loop system. If there exists two constants $M>0$, $\kappa >0$, and the NN mapping vector $\bm{\pi_{\theta}}$, such that for any initial {\color{black}state} $x(0)$ and $\widehat{x}(0)$, the closed-loop system satisfies that for all $t>0$,
    $$
    \|x(t)-T\widehat{x}(t)\|_2\leq M\exp^{-\kappa t}\{\|x(0)\|_2+\|\widehat{x}(0)\|_2\},
    $$
    then the system \eqref{eq: general systems} is called {\color{black}\textbf{neural exponentially observable}}.
    \end{definition}
    

    In this paper, the canonical observation problems for three specific dynamical models of system~\eqref{eq: general systems} are taken into account, including the linear systems without uncertainty, the \textcolor{black}{integrator chain} systems and the MIMO nonlinear systems (consisting of a linear dynamic part and the general uncertainty). Based on the above definitions, we would like to post the question: \textit{under what conditions are these systems neural observable?}
    
    \subsection{Neural Observers for Linear Systems}
    \tikzstyle{block} = [draw, fill=black!20!red!20!white, rectangle, minimum height=3em, minimum width=6em, very thick]
    \tikzstyle{block_NN} = [draw, fill=black!20!blue!20!white, rectangle, minimum height=3em, minimum width=6em, very thick]
    \tikzstyle{sum} = [draw, fill=black!20!blue!40!white, circle, node distance=1cm, very thick]
    \tikzstyle{input} = [coordinate]
    \tikzstyle{output} = [coordinate]
    
    We first consider the following continuous-time LTI system without uncertainty, that is a typical case in \eqref{eq: general systems}:
    \begin{equation}
        \left\{
        \begin{aligned}
        \dot{x}&=Ax+Bu,\ x(0)=x_0,\\
        y&=Cx,
        \end{aligned}
        \right.
        \label{eq: linear systems}
    \end{equation}
    where $A\in \mathcal{L}(\mathbf{R}^{n_s})$ and $B\in \mathcal{L}(\mathbf{R}^{n_s},\mathbf{R}^{n_u})$ are known system matrices. For the neural observable problem, we employ the standard assumption.
    \begin{assumption}
    $(A, B)$ is controllable, and $(C, A)$ is observable.
    \label{ap: control-observe condition}
    \end{assumption}
    
    Due to the availability of system matrices $(A,B,C)$, we can construct a neural observer corresponding to \eqref{eq: linear systems}, which is consistent with Remark \ref{rm: keys for NO}, as follows:
    \begin{equation}
    \left\{
        \begin{aligned}
            \dot{\widehat{x}}&=A\widehat{x}+Bu+\pi_{\theta_1}(y-\widehat{y}),\ \widehat{x}(0)=\widehat{x}_0,\\
            \widehat{y}&=C\widehat{x},
        \end{aligned}
    \right.
    \label{eq: nos for linear systems}  
    \end{equation}
    where $\pi_{\theta_1}(\cdot)$ is {\color{black}an NN} mapping. The concrete neural observation diagram for the system \eqref{eq: linear systems} is shown in Fig.~\ref{fig: NN-continuous}.  
    Accordingly, we propose the following intuitive questions:
    \begin{figure}[t]
        \centering
        \scalebox{0.7}{
            \begin{tikzpicture}[auto,  node distance=2cm, >=latex']
                \node [input] (input) {};
                \node [block, right of=input,  node distance=3cm] (nominal plant) {$\dot{x}(t)=Ax(t)+Bu(t)$};
                \draw [draw,-latex,very thick] (input) -- node [name=u] {$u(t)$} (nominal plant);
                \node [block, right of=nominal plant, node distance=4.5cm] (nominal output) {$y=Cx$};
                \draw  [-latex,very thick] (nominal plant) -- node[name=x] {$x(t)$} (nominal output);
                \node [output, right of=nominal output, node distance=3cm] (output) {$y(t)$};
                \draw [-latex,very thick] (nominal output) -- node [name=y] {$y(t)$}(output);

                \node [sum, below of=y, node distance=2.5cm] (sum) {$\bm{+}$};
                \draw [-latex,very thick] (y) -- (sum);
                \node [block_NN, below of=nominal plant, node distance=2.25cm] (NN) {\NNsimple};
                \draw [-latex,very thick] (sum) --node  [near end] {$y(t)-\widehat{y}(t)$} (NN);
                \draw (NN)  node [yshift=-0.45cm, xshift=-1.5cm] {$\pi_{\theta_1}$};

                \node [block_NN, below of=NN, node distance=2.25cm] (observer plant) {$\dot{\widehat{x}}(t)=A\widehat{x}(t)+Bu(t)$};
                \draw [-latex,very thick] (NN) -- (observer plant);
                \node [block_NN, right of=observer plant, node distance=4.5cm] (observer output) {$\widehat{y}=C\widehat{x}$};
                \draw  [-latex,very thick] (observer plant) -- node[name=observer_x] {$\widehat{x}(t)$} (observer output);
                \draw [-latex,very thick] (observer output)  -| node [name=observer_y, near start] {$\widehat{y}(t)$} (sum);
                \draw [-latex,very thick] (u) |- (observer plant);
                \draw (sum) node [yshift=-0.5cm, xshift=0.25cm]  {$\bm{-}$};

            \end{tikzpicture}
        }
        \caption{The diagram of neural observers for LTI (continuous) systems}
        \label{fig: NN-continuous}
    \end{figure}
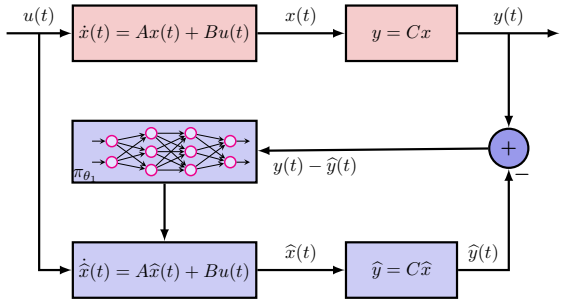
    \begin{problem}
    Under Assumption \ref{ap: control-observe condition}, what is the necessary condition for the system \eqref{eq: linear systems} to be {\color{black}neural exponentially observable}?
    \label{pro: no for linear systems}
    \end{problem}
    
    {\color{black}We note that in analogy with classical control problems, \textit{if the system \eqref{eq: linear systems} is {\color{black}neural exponentially observable}, the {\color{black}estimated state} $\widehat{x}(t)$ can be utilized to design a feedback controller for system \eqref{eq: linear systems}.}} Inspired by NN controllers for discrete-time LTI systems in \cite{yin2022stability,pauli2021offset}, we construct the following observer-based NN controller:
    \begin{equation}
        u(t) = \pi_{\theta_2}(\widehat{x}).
        \label{eq: controller for linear systems}
    \end{equation}
    Consequently, a second question for system \eqref{eq: linear systems} is raised:
    \begin{problem}
    By applying the control law $u(t)=\pi_{\theta_2}(\widehat{x})$, the problem is to find a suitable NN mapping $\pi_{\theta_2}(\cdot)$ such that $\lim_{t\rightarrow \infty} \|x(t)\|_2=0$, in the meanwhile the system \eqref{eq: linear systems} is {\color{black}neural exponentially observable} under the given Assumption \ref{ap: control-observe condition}.
    \label{prop: control problem for linear systems}
    \end{problem}
    
    The results for the above questions are considered as the most fundamental ones, which could be served as the baselines in the following sections. 
    
    \subsection{Neural Observers for Integrator Chain Nonlinear Systems}  
   We consider a class of SISO uncertain systems described by the following differential equation with an order of $n$:
    \begin{equation}
        \left\{
        \begin{aligned}
        x^{(n)}(t)&=\mathcal{F}(t,x(t),\cdots,x^{(n-1)}(t),w(t))+bu(t),\\
        y(t)&=x(t),
        \end{aligned}
        \right.
        \label{eq: integral-chain nonlinear systems}
    \end{equation}
    where $\mathcal{F}(\cdot)$ is an unknown and continuously differentiable function, and $b$ is a known constant. \textcolor{black}{The above system \eqref{eq: integral-chain nonlinear systems} is an integral-chain system that can be rewritten as a controller canonical form:
    \begin{equation}
        \left\{
        \begin{aligned}
            \dot{\bm{x}}&=\mathcal{A}\bm{x}+\mathcal{B}(\mathcal{F}(t,\bm{x},w)+bu),\ \bm{x}(0)=\bm{x}_0,\\
            y&=c\bm{x},
        \end{aligned}
        \right.
        \label{eq: integral-chain nonlinear systems_2}
    \end{equation}
    where $\bm{x} = [x_1,\cdots,x_n]^{\top}$, $\mathcal{A} = (a_{ij})_{n\times n}$ is defined by 
    \begin{equation*}
        \begin{aligned}
            a_{ij}&=\left\{\begin{array}{l}
        1,\ i+1=j  \\
        0,\ \text{else}
        \end{array}\right.,\ c=[1,0, \cdots, 0], \\
            \mathcal{B}&=[0, \ldots, 0,1]^{\top}. 
        \end{aligned}
    \end{equation*}
    Note that $(\mathcal{A},\mathcal{B},c)$ is a canonical form representation of a chain of $n$ integrators.} \textcolor{black}{We note that if $\mathcal{F}(\cdot)$ in the integral-chain system is known, one can prove that \eqref{eq: integral-chain nonlinear systems_2} is a flat and controllable system \cite{fliess1995flatness}. When the order $n=2$, the model \eqref{eq: integral-chain nonlinear systems} can describe most of the common physical systems via Newton's second law, including the inverted pendulum model shown in Section \ref{sec: Numerical Experiments}.} And due to the differentiability of $\mathcal{F}(\cdot)$, the system \eqref{eq: integral-chain nonlinear systems_2} is a case of \eqref{eq: general systems}. Then, we make some basic assumptions for nonlinear systems \eqref{eq: integral-chain nonlinear systems}.

    From the Remark \ref{rm: keys for NO}, we can regard the matrix $\mathcal{A}$ and $c$ as the knowledge that is used to describe the corresponding neural observer for systems \eqref{eq: integral-chain nonlinear systems}:
    \begin{equation}
        \left\{
        \begin{aligned}
            \dot{\widehat{x}}_i&=\widehat{x} _{i+1}+\epsilon^{n-i}\pi_{\theta_i}(\epsilon^{-n}(y-\widehat{y})),\widehat{x}_i(0)=\widehat{x}_{i,0},\\ \qquad &\qquad\qquad\qquad\qquad\qquad\qquad\qquad i=1,\cdots,n-1, \\
            \dot{\widehat{x}}_n&=\widehat{x} _{n+1}+\pi_{\theta_n}(\epsilon^{-n}(y-\widehat{y}))+bu,\widehat{x}_n(0)=\widehat{x}_{n,0},\\
            \dot{\widehat{x}}_{n+1}&=\epsilon^{-1}\pi_{\theta_{n+1}}(\epsilon^{-n}(y-\widehat{y})),\widehat{x}_{n+1}(0)=\widehat{x}_{n+1,0},\\
            \widehat{y}&=c\widehat{\bm{x}},
        \end{aligned}  
        \right.
        \label{eq: nos for integral-chain nonlinear systems}
    \end{equation}
    where $\epsilon$ is a positive constant, $\widehat{\bm{x}}=\left[\widehat{x}_1,\cdots,\widehat{x}_n\right]^{\top}$.
    
    \begin{assumption}[\cite{guo2013on,guo2016active}]
    (1) Firstly, there exists a continuous function $\chi(x,w)$ such that $\sup_{t \in \mathcal{T}_{\geq 0}}\|(\mathcal{F},\nabla \mathcal{F})\|_2\leq\chi(x,w)$. 
    (2) Secondly, there exists a bounded control $u(t)$ such that $\sup_{t \in \mathcal{T}_{\geq 0}} \{\|\bm{x}(t)\|_2+|u(t)|\}<\infty$.
    \label{ap: the uncertainty for integral-chain}
    \end{assumption}
    
    We note that the first assumption imposed the differentiability of noise $w$ and uncertainty $\mathcal{F}$ with respect to time. Based on the neural observer, we note that the simplest way to satisfy Assumption \ref{ap: the uncertainty for integral-chain} is to design the bounded control in the linear form $u(t)=\frac{\rho \sum_{i=1}^n k_i \operatorname{sat}_{M_i}(\rho^{n-i}\widehat{x}_i(t))-\operatorname{sat}_{M_{n+1}}(\widehat{x}_{n+1}(t))}{b}$ with parameters of $\rho,k_i,M_i$, where the control gain $K=[k_1,\cdots,k_n]$ is designed by the Hurwitz matrix $\mathcal{A}+c_0K$ with $c_0=[0,\cdots,1]^{\top}$, the details of which can be found in \cite{zhao2016active}.
    
    Therefore, we informally introduce the observation problem for systems \eqref{eq: integral-chain nonlinear systems}:
    \begin{problem}
    Consider the neural observer-based closed system \eqref{eq: integral-chain nonlinear systems}, \eqref{eq: nos for integral-chain nonlinear systems}, with the Assumption \ref{ap: the uncertainty for integral-chain}, we need to investigate the necessary conditions for the underlying result:
    \begin{itemize}
        \item the system \eqref{eq: integral-chain nonlinear systems} is neural observable.
    \end{itemize}
    \label{pro: dc for integral-chain systems}
    \end{problem}
    \begin{remark}
        We need to point out that the formulation of the neural observer \eqref{eq: nos for integral-chain nonlinear systems} resembles the extended state observer (ESO) in active disturbance rejection control \cite{guo2013on,zhao2017nonlinear}. But in fact, the construction of nonlinear functions in ESO is complicated in industrial processes. Hence, due to the approximating capability of NNs, we can take advantage of this property to relieve these pressures.
    \end{remark}
    \subsection{Neural Observers for MIMO systems}
    \textcolor{black}{As another case of system \eqref{eq: general systems}, the following MIMO nonlinear system (composed of a linear dynamic and general uncertainty) is taken into account.
    \begin{equation}
        \left\{
        \begin{aligned}
           \dot{x}&=Ax+Bu+B_w\mathcal{K}(x,w,t),\ x(0)=x_0,\\
            y&=Cx,
        \end{aligned}
        \right.
        \label{eq: nonintegral-chain systems}
    \end{equation}
    where $A$ and $B$ are defined in the same way as those in \eqref{eq: linear systems}. $\mathcal{K}\in C^1(\mathbf{R}^{n_s}\times\mathbf{R}^{n_d}\times\mathbf{R};\mathbf{R}^{n_q})$ represents the uncertainty with respect to $x(t)$, $w(t)$, and $t$. $B_w\in \mathcal{L}(\mathcal{R}^{n_q},\mathcal{R}^{n_s})$ is a known matrix.
    \begin{remark}
    Since system \eqref{eq: integral-chain nonlinear systems_2} is an affine control system,
    system \eqref{eq: integral-chain nonlinear systems_2} is one case of \eqref{eq: nonintegral-chain systems}. In other words, \eqref{eq: nonintegral-chain systems} is a more general formulation compared to \eqref{eq: integral-chain nonlinear systems}. We also note that \eqref{eq: nonintegral-chain systems} is not restricted to the integral chain form and may subject to the mismatched uncertainty and disturbances \cite{li2012generalized}.
    \end{remark}} 
    \begin{assumption}
    \begin{itemize}
        \item[(1)] \textcolor{black}{$A$, $C$ and $B_w$ satisfy an \textit{extending observable condition}, i.e.,
        $$
        (\mathbf{C},\mathbf{A})\triangleq\left(\left[C,O\right],\left[\begin{array}{cc}
        A & B_w\\
        O & O
        \end{array}\right]\right)\ \text{is observable};
        $$
        }
        \item[(2)] $\sup_{t \in \mathcal{T}_{\geq 0}}(\|\mathcal{K}\|_2,\|\nabla \mathcal{K}\|_2)\leq\varpi(x,w)$, where $\varpi(x,w)\in C(\mathbf{R}^{n_s+1};\mathbf{R}) $; 
        \item[(3)] \textcolor{black}{there exists a bounded control law $u=u(t)$ such that the state $x(t)$ is bounded.}
    \end{itemize}
    \label{ap: the uncertainty for nonintegral-chain systems}
    \end{assumption}
    \textcolor{black}{Now, we give an example to illustrate the \textit{extending observable condition} and a necessary condition of the extending observability.
    \begin{example}
        We consider that $A=\left[\begin{array}{cc}
        0 & 1\\
        1 & 1
        \end{array}\right]$, $B_w=[1,0]^{\top}$ and $C=[1,0]$. It is easy to check that 
        $$
        \operatorname{r}\left[\begin{array}{c}
        \mathbf{C}^{\top}, (\mathbf{CA})^{\top}, (\mathbf{CA^2})^{\top}
        \end{array}\right]=3 \Longleftrightarrow (\mathbf{C},\mathbf{A})\ \text{is observable}.
        $$
    \end{example}
    \begin{proposition}
        $(C,A)$ is observable if $A$, $C$, $B_w$ satisfy the \textit{extending observable} condition defined in Assumption \ref{ap: the uncertainty for nonintegral-chain systems}.
        \label{prop: ness-cond for extending observable}
    \end{proposition}
    \begin{proof}
        The proof is given in Appendix \uppercase\expandafter{\romannumeral1}.
    \end{proof}}
    \textcolor{black}{According to Remark \ref{rm: keys for NO}, we {\color{black}hold} $(A,B,C)$ and $B_w$ as the knowledge to design the neural observer, as shown below:
    \begin{equation}
        \left\{
        \begin{aligned}
            \dot{\widehat{x}}_1&=B_w\widehat{x}_2 +A\widehat{x}_1+Bu+ \pi_{\theta_1}(\epsilon^{-1}(y-\widehat{y})),\ \widehat{x}_1(0)=\widehat{x}_{1,0},\\
            \dot{\widehat{x} }_2&=\epsilon^{-1}\pi_{\theta_2}(\epsilon^{-1}(y-\widehat{y})),\ \widehat{x}_2(0)=\widehat{x}_{2,0},\\
            \widehat{y}&=\mathbf{C}\bm{\widehat{x}},
        \end{aligned}  
    \right.
    \label{eq: nos for nonintegral-chain linear systems}
    \end{equation}
    where $\epsilon$ is positive, $\bm{\widehat{x}}=[\widehat{x}^{\top}_1,\widehat{x}^{\top}_2]^{\top}$.} In the end, a problem is raised accordingly:
    
    \begin{problem}
    We need to seek out the necessary condition for system \eqref{eq: nonintegral-chain systems} to be neural observable under Assumption \ref{ap: the uncertainty for nonintegral-chain systems}.
    \label{pro: d for nonintegral-chain systems}
    \end{problem}
\section{NN Representation and NN mapping vector}
    \label{sec: NN Representation and NN mapping vector}
    We will present the main theorems of this work in a later section: Theorems \ref{theo: LMI for LTI without control}-\ref{theo: LMI for nonintegral-chain systems}, which directly solve each of the problems mentioned above. As a necessary prelude, however, we introduce the following two definitions. The first definition is about the isolation of nonlinear activation function from the linear operation of NNs defined in \eqref{eq: resnets} and QCs for activation functions, similarly done in \cite{yin2022stability,fazlyab2022safety} and \cite{fazlyab2019efficient}, respectively. In the second, we define the concept of NN mapping vector and QCs for NN mapping vector.
\subsection{NN Isolation and QCs for Single NN mapping}
    For a specific NN mapping $\pi_{\theta}$ and the input $x\in\mathbf{R}^{\text{Input}}$, we define $w^{0}=x$ and $\xi^i=W^iw^{i-1}$, $w^i=\sigma^{[i]}(\xi^i),\ i=1,\cdots,L$. By collecting the input and output of all activation functions, we denote two $n_{\sigma}$ dimensional vectors $\xi_{\sigma}$ and $w_{\sigma}$ as follows:
    $$
    \xi_{\sigma}\triangleq\left[\begin{array}{c}
    \xi^{1} \\
    \vdots \\
    \xi^{L}
    \end{array}\right],  \ 
    w_{\sigma}\triangleq\left[\begin{array}{c}
    w^{1} \\
    \vdots \\
    w^{L}
    \end{array}\right].
    $$
    Then, by recalling that $\pi_{\theta}(x)=W^{L+1}w^{L}+W^{L+2}\pi_{\theta}^{[0]}(x)$ and $\pi_{\theta}^{[0]}(x)=x$, the NN mapping $\pi_{\theta}$ can be rewritten into
    \begin{equation}
    \left[
        \begin{array}{c}
            \pi_{\theta}\\
            \xi^1 \\
            \vdots \\
            \xi^L
        \end{array}
    \right]=
    \left[\begin{array}{c|cccc}
        W^{L+2} & O & O & \cdots & W^{L+1} \\
        \hline W^{1} & O & \cdots & O & O \\
        O & W^{2} & \cdots & O & O \\
        \vdots & \vdots & \ddots & \vdots & \vdots \\
        O & O & \cdots & W^{L} & O
        \end{array}\right]
    \left[
        \begin{array}{c}
            x\\
            w^1 \\
            \vdots \\
            w^L
        \end{array}
    \right],
    \label{eq: NN representation 0}
    \end{equation}
    which can be abbreviated in the following formulation:
    \begin{equation}
        \left[
            \begin{array}{c}
                \pi_{\theta}(x)\\
                \xi_{\sigma}
            \end{array}
        \right]=
        \left[
            \begin{array}{cc}
                N_{\pi x} & N_{\pi w} \\
                N_{\xi x} & N_{\xi w}
            \end{array}
        \right]
        \left[
            \begin{array}{c}
                x\\
                w_{\sigma}
            \end{array}
        \right]
        \label{eq: NN representation}
    \end{equation}
    Now, we define {\color{black}the two following linear mappings}:
    \begin{equation}
        R_{\pi} \triangleq
        \left[
            \begin{array}{cc}
                I&O \\
                N_{\pi x}&N_{\pi w}
                \end{array}
        \right],\quad
        R_{\xi} \triangleq
        \left[
            \begin{array}{cc}
                N_{\xi x}&N_{\xi w}\\
                O&I
                \end{array}
        \right],
        \label{eq: R_pi and R_xi}
    \end{equation}
    then derive the corresponding linear transformations for $[x^{\top},w_{\sigma}^{\top}]^{\top}$:
    $$
    \left[
        \begin{array}{c}
            x\\
            \pi_{\theta}(x)
        \end{array}
    \right]=R_{\pi}
    \left[
        \begin{array}{c}
            x\\
            w_{\sigma}
        \end{array}
    \right],\quad
    \left[
        \begin{array}{c}
            \xi_{\sigma}\\
            w_{\sigma}
        \end{array}
    \right]=R_{\xi}
    \left[
        \begin{array}{c}
            x\\
            w_{\sigma}
        \end{array}
    \right].
    $$
    To avoid confusion, we must emphasize that two identical matrices $I$ in $R_{\pi}$ and $R_{\xi}$ belong to different spaces of linear mappings, $\mathcal{L}(\mathbf{R}^{\text{Input}})$ and $\mathcal{L}(\mathbf{R}^{n_{\sigma}})$, respectively. 
    \begin{remark}
    Due to the existence of a shortcut connection shown in Fig. \ref{fig: resnets}, the matrix $N_{\pi x}$ in \eqref{eq: R_pi and R_xi} is a non-zero matrix that makes a difference with \cite{yin2022stability}. In addition, the non-zero matrix $N_{\pi x}$ plays a crucial role in neural observability, which would be certified in Proposition \ref{prop: existence of solution of LMI}.
    \end{remark}
    
    Next, we deal with another thorny difficulty in analyzing NNs, which is the composition of nonlinear activation functions. The key is to remove the non-linearity of activation functions but preserve some geometrical properties.
    
    Consider the activation function $\sigma(\cdot)\in C(\mathbf{R}|\sigma(0)=0)$, then, the function is said to be sector bounded in sector $[\alpha,\beta]$ with $\alpha\leq \beta< \infty$ if the following inequality holds for all $s\in\mathbf{R}$:
    $$
    \left(\sigma(s)-\alpha s \right) \left(\beta s-\sigma(s)\right) \geq 0.
    $$
    Intuitively, the above inequality implies that the function $y = \sigma(s)$ lies in the open region of $y=\alpha s$, $y=\beta s$, and the origin. For the sector bounded, as mentioned earlier, the nonlinear functions commonly used in practice \cite{zhao2017nonlinear} are of the following form:
    $$
    \operatorname{fal}(s, \gamma, \delta)=\left\{\begin{array}{l}
            |s|^{\gamma} \operatorname{sgn}(s),\ |s|>\delta, \\
            \frac{s}{\delta^{1-\gamma}},\ |s| \leq \delta,
            \end{array}\right.\ 
    $$
    which also satisfies the sector boundedness illustrated in Fig. \ref{fig: sector on fal and G}.
    \begin{figure}[t]
        \centering
        \scalebox{0.7}{
           \begin{tikzpicture}
                \centering
                \draw[<->, ultra thick](\num,0)node[below] {$s$}--(0,0)--(0,\num-1)node[above] {$\operatorname{y}(s)$};
                \draw[ultra thick](-\num,0)--(0,0)--(0,-\num+2);
                \draw[heavy_blue,domain=-0.5:0.5,ultra thick] plot(\x,{1/(sqrt(0.5))*\x});
                \draw [heavy_blue,domain=0.5:4,ultra thick,samples=100] plot(\x,{(\x)^(0.5)})node[above]{$y=\operatorname{fal}(s, \gamma, \delta)$};
                \draw [heavy_blue,domain=-4:-0.5,ultra thick,samples=100] plot(\x,{-(-\x)^(0.5)});
                \draw [heavy_purple,domain=-1:1.5,ultra thick,samples=50] plot(\x,{((2*\x)+0.01)})node[below,right]{$y=\operatorname{\beta}s$};
                \draw [rose_red ,domain=-4:4,ultra thick,samples=50] plot(\x,0.05)node[above]{$y=\operatorname{\alpha}s$};
                \draw [heavy_blue,ultra thick,densely dashed] (0.5,0)node[below]{$\delta$}--(0.5,0.7);
                \draw [heavy_blue,ultra thick,densely dashed] (-0.5,0)node[above]{$-\delta$}--(-0.5,-0.7);
            \end{tikzpicture} 
        }
        \caption{Sector boundedness for $\operatorname{fal}(s, \gamma, \delta)$: The function $\operatorname{fal}(s, \gamma, \delta)$ is sector bounded with $[\alpha,\beta]$, where we set $\alpha=0$ and $\beta\geq \delta^{\gamma-1}$.}
        \label{fig: sector on fal and G}
    \end{figure}
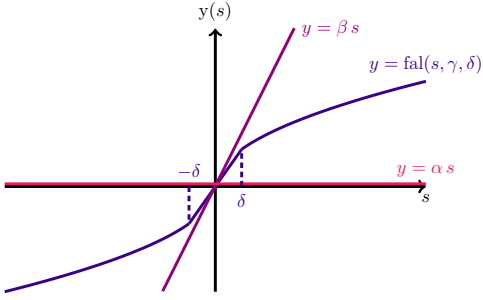
    In other words, when we take the NN in the neural observer as a single layer, with no shortcut connection ($W^{L+2}=O$), and use the $\operatorname{fal}(s, \gamma, \delta)$ function as the activation function, in this sense, then the extended state observers with the $\operatorname{fal}(s, \gamma, \delta)$ function can be included in neural observers. 
    
    Next, the activation functions $\sigma(\cdot)$ at each hidden layer are sector bounded in sector $[\alpha_i,\beta_i],i=1,\cdots,n_{\sigma}$, respectively. By denoting sector vectors $\alpha_{{\sigma}}=[\alpha_1,\cdots,\alpha_{n_{\sigma}}]$ and $\beta_{{\sigma}}=[\beta_1,\cdots,\beta_{n_{\sigma}}]$, the QCs for one NN mapping are provided as follows:
    \begin{lemma}[\cite{yin2022stability}]
    Let $\alpha_{{\sigma}}, \beta_{{\sigma}} \in \mathbf{R}^{n_{\sigma}}$ be defined above with $\alpha_{n_{\sigma}} \leq \beta_{n_{\sigma}}$. If $\lambda_{\sigma} \in \mathbf{R}^{n_{\sigma}}$ and $\lambda_{\sigma} \geq \bm{0}_{n_{\sigma}}$, then:
    \begin{equation*}
        \begin{gathered}{
            \left[\begin{array}{c}
            \xi_{\sigma} \\
            w_{\sigma}
            \end{array}\right]^{\top} \Psi_{\sigma}^{\top} M_{\sigma}(\lambda_{\sigma}) \Psi_{\sigma}\left[\begin{array}{c}
            \xi_{\sigma} \\
            w_{\sigma}
            \end{array}\right] \geq 0
            }
        \end{gathered}    
    \end{equation*}
    where $\bm{0}_{\sigma}$ is a zero vector,
    \begin{equation}
        \begin{aligned}
            \Psi_{\sigma}\triangleq&\left[\begin{array}{cc}\operatorname{diag}\left(\beta_{\sigma}\right) & -I \\ -\operatorname{diag}\left(\alpha_{\sigma}\right) & I\end{array}\right],\\
            M_{\sigma}(\lambda_{\sigma})\triangleq&\left[\begin{array}{cc}O & \operatorname{diag}(\lambda_{\sigma}) \\ \operatorname{diag}(\lambda_{\sigma}) & O\end{array}\right].
        \end{aligned}
        \label{eq: Psi_sigma and M_sigma}
    \end{equation}
    \label{lem: QC for one mapping}
    \end{lemma}

\subsection{NN Isolation and QCs for NN mapping vector}
    Whereafter, we try to isolate the non-linearity of NN mapping vector $\bm{\pi_{\theta}}(x)=[\pi^{\top}_{\theta_1}(x),\cdots,\pi^{\top}_{\theta_{K}}(x)]^{\top}$ shown in Fig. \ref{fig: NN mapping vector}, which consists of $K$ NN mappings with different parameters $\theta_i$. 
    \tikzstyle{block} = [draw, fill=blue!20, rectangle, minimum height=3pt, minimum width=6pt]
    \tikzstyle{input} = [coordinate]
    \tikzstyle{output} = [coordinate]
    \tikzstyle{layer} = [coordinate]
    \tikzstyle{pinstyle} = [pin edge={to-, very thick, black}]
    
    \begin{figure}[t]
        \centering
        \scalebox{0.8}{
            \begin{tikzpicture}[auto, >=latex']
                \node [input] (input_1) at (0, 1.5) {};
                \node [block, right of=input_1, node distance=4cm] (NN_1) {\NNsimple};
                \draw [pinstyle,-latex] (input_1) --node[above]{$T_1^1$} (NN_1);
                \node [right of=NN_1, node distance=3cm] (output_1) {$\pi_{\theta_1}(x)$};
                \draw [pinstyle,-latex] (NN_1) --node[above]{$T_1^2$} (output_1);
                \draw (NN_1)  node [yshift=-0.45cm, xshift=-1.5cm] {$\pi_{\theta_1}$};
                \draw (NN_1)  node [yshift=0.8cm] {$L_1$ layers};

                \node [input] (input_2) at (0, -1.5) {};
                \node [block, right of=input_2, node distance=4cm] (NN_2) {\NNsimple};
                \draw [pinstyle,-latex] (input_2) --node[above]{$T_K^1$} (NN_2);
                \node [right of=NN_2, node distance=3cm] (output_2) {$\pi_{\theta_K}(x)$};
                \draw [pinstyle,-latex] (NN_2) --node[above]{$T_K^2$} (output_2);
                \draw (NN_2)  node [yshift=-0.45cm, xshift=-1.5cm] {$\pi_{\theta_K}$};
                \draw (NN_2)  node [yshift=-0.8cm] {$L_K$ layers};

                \node (input) at (0,0) {$x$};
                \draw [dashed, very thick, -latex] (input) -- node[name=x_1] {\textcolor{blue}{input}} (input_1);
                \draw [dashed, very thick, -latex] (input) -- node[name=x_2,left] {\textcolor{blue}{input}} (input_2);

                \node (output) at (7,0) {$\bm{\pi_{\theta}}(x)=[\pi^{\top}_{\theta_1}(x),\cdots,\pi^{\top}_{\theta_{K}}(x)]^{\top}$};
                \node (series) at (4,.3) {$\vdots$};
                \node (series) at (4,-.3) {$\vdots$};
                \draw [dashed, very thick, -latex] (output_1) -- (output);
                \draw [dashed, very thick, -latex] (output_2) -- (output);
            \end{tikzpicture}
        }
        \caption{The diagram of NN mapping vector}
        \label{fig: NN mapping vector}
    \end{figure}
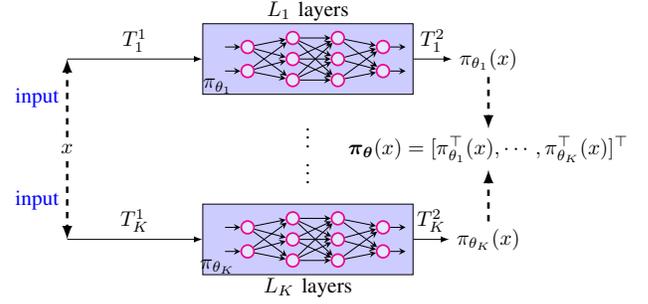
   
    We denote $w_k^{0}=x$ and $\xi_k^i=W_k^iw_k^{i-1}$, $w_k^i=\sigma^{[i]}(\xi_k^i),\ i=1,\cdots,L_k,\ k=1,\cdots,K$, and two $n_{\sigma_k}\triangleq \sum_{i=1}^{L_k} n_i$ dimensional vectors $\xi_{\sigma_k}$ and $w_{\sigma_k}$ as follows:
    $$
    \xi_{\sigma_k}\triangleq\left[\begin{array}{c}
    \xi_k^{1} \\
    \vdots \\
    \xi_k^{L_k}
    \end{array}\right],  \ 
    w_{\sigma_k}\triangleq\left[\begin{array}{c}
    w_k^{1} \\
    \vdots \\
    w_k^{L_k}
    \end{array}\right].
    $$
    With the help of \eqref{eq: NN representation}, we derive the following transformation 
    \begin{equation*}
        \left[
            \begin{array}{c}
                \pi_{\theta_k}(x)\\
                \xi_{\sigma_k}
            \end{array}
        \right]=
        \underbrace{\left[
            \begin{array}{cc}
                N_{\pi_k x} & N_{\pi_k w_k} \\
                N_{\xi_k x} & N_{\xi_k w_k}
            \end{array}
        \right]}_{N_k}
        \left[
            \begin{array}{c}
                x\\
                w_{\sigma_k}
            \end{array}
        \right],
    \end{equation*}
    $$
    N_k=
    \left[\begin{array}{c|cccc}
        T^2_k W_k^{L_k+2} T^1_k & O & O & \cdots & T^2_kW_k^{L_k+1} \\
        \hline W_k^{1}T^1_k & O & \cdots & O & O \\
        O & W_k^{2} & \cdots & O & O \\
        \vdots & \vdots & \ddots & \vdots & \vdots \\
        O & O & \cdots & W_k^{L_k} & O
        \end{array}\right].
    $$
    We suppose $\xi^{K}_{\sigma}=[\xi^{\top}_{\sigma_1},\cdots,\xi^{\top}_{\sigma_K}]^{\top}$ and $w^{K}_{\sigma}=[w^{\top}_{\sigma_1},\cdots,w^{\top}_{\sigma_K}]^{\top}$, and have
    \begin{equation}
    \left[
        \begin{array}{c}
            \bm{\pi_{\theta}}(x)\\
            \xi^{K}_{\sigma}
        \end{array}
    \right]=
    \left[
        \begin{array}{cc}
            \widehat{N_{\bm{\pi}x}}  & \widehat{N_{\bm{\pi} w}}  \\
            \widehat{N_{\xi x}}  & \widehat{N_{\xi w}} 
        \end{array}
    \right]
    \left[
        \begin{array}{c}
            x\\
            w^{K}_{\sigma}
        \end{array}
    \right],
    \label{eq: NN representation_2}
    \end{equation}
    where the block matrices equal to
    $$
    \widehat{N_{\bm{\pi}x}}=\left[\begin{array}{c}
    N_{\pi_1 x} \\
    \vdots \\
    N_{\pi_K x}
    \end{array}\right],\ 
    \widehat{N_{\bm{\pi} w}}=\operatorname{diag}\left(N_{\pi_1 w_1},\cdots,N_{\pi_K w_K}\right),
    $$
    $$
    \widehat{N_{\xi x}}=\left[\begin{array}{c}
    N_{\xi_1 x} \\
    \vdots \\
    N_{\xi_K x}
    \end{array}\right],\ 
    \widehat{N_{\xi w}}=\operatorname{diag}\left(N_{\xi_1 w_1},\cdots,N_{\xi_K w_K}\right).
    $$
    Then, it is easy to verify that the following transformations derived from \eqref{eq: NN representation_2} are held on:
    $$
    \left[
        \begin{array}{c}
            x\\
            \bm{\pi_{\theta}}(x)
        \end{array}
    \right]=\widehat{R_{\bm{\pi}}}
    \left[
        \begin{array}{c}
            x\\
            w^{K}_{\sigma}
        \end{array}
    \right],\quad
    \left[
        \begin{array}{c}
            \xi^{K}_{\sigma}\\
            w^{K}_{\sigma}
        \end{array}
    \right]=\widehat{R_{\xi}}
    \left[
        \begin{array}{c}
            x\\
            w^{K}_{\sigma}
        \end{array}
    \right],
    $$
    where 
    \begin{equation}
        \widehat{R_{\bm{\pi}}} \triangleq
        \left[
            \begin{array}{cc}
                I&O \\
                \widehat{N_{\bm{\pi} x}}&\widehat{N_{\bm{\pi} w}}
                \end{array}
        \right],\quad
        \widehat{R_{\xi}} \triangleq
        \left[
            \begin{array}{cc}
                \widehat{N_{\xi x}}&\widehat{N_{\xi w}}\\
                O&I
                \end{array}
        \right].
        \label{eq: R_pi and R_xi 2}
    \end{equation}
    Afterward, by gathering each sector vector $\alpha_{\sigma_k},\ \beta_{\sigma_k}$ from $\pi_{\theta_k}(x)$ and making $n_{\sigma}^K=\sum_{k=1}^K n_{\sigma_k}$, $\alpha_{\sigma}^K=[\alpha^{\top}_{\sigma_1},\cdots,\alpha^{\top}_{\sigma_K}]^{\top}$, $\beta_{\sigma}^K=[\beta^{\top}_{\sigma_1},\cdots,\beta^{\top}_{\sigma_K}]^{\top}$, we take a vector $\lambda^K_{\sigma}\in\mathbf{R}^{n_{\sigma}^K}$ with positive components and the following matrices with the parameter $K$ 
    \begin{equation}
        \begin{aligned}
            \mathbf{\Psi}_{\sigma}(\bm{\alpha},\bm{\beta};K)\triangleq&\left[\begin{array}{cc}\operatorname{diag}\left(\beta^K_{\sigma}\right) & -I \\ -\operatorname{diag}\left(\alpha^K_{\sigma}\right) & I\end{array}\right],\\
            \mathbf{M}_{\sigma}(\lambda^K_{\sigma};K)\triangleq&\left[\begin{array}{cc}O & \operatorname{diag}(\lambda^K_{\sigma}) \\ \operatorname{diag}(\lambda^K_{\sigma}) & O\end{array}\right].
        \end{aligned}
        \label{eq: Psi_sigma and M_sigma 2}
    \end{equation}
    {\color{black} Later, $\mathbf{\Psi}_{\sigma}(\bm{\alpha},\bm{\beta};K)$ and $ \mathbf{M}_{\sigma}(\lambda^K_{\sigma};K)$ are simply denoted as $\mathbf{\Psi}(K)$ and $\mathbf{M}(K)$, respectively.} Then, it suffices to show the following QC for NN mapping vector:
    \begin{lemma}[$K$ Law of Quadratic Constraint]
    For any $\lambda^K_{\sigma}\geq \bm{0}_{n_{\sigma}^K}$, $\xi_{\sigma}^K$ and $w_{\sigma}^K$ defined in \eqref{eq: NN representation_2}, matrices $\mathbf{\Psi}_{\sigma}(\bm{\alpha},\bm{\beta};K)$ and $ \mathbf{M}_{\sigma}(\lambda^K_{\sigma};K)$ defined in \eqref{eq: Psi_sigma and M_sigma 2}, the QC is held for NN mapping vector $\bm{\pi_{\theta}}$ consisting of $K$ NN mappings
    $$
    \begin{gathered}{
        \left[\begin{array}{c}
        \xi^K_{\sigma} \\
        w^K_{\sigma}
        \end{array}\right]^{\top} \mathbf{\Psi}(K)^{\top} \mathbf{M}(K) \mathbf{\Psi}(K)\left[\begin{array}{c}
        \xi^K_{\sigma} \\
        w^K_{\sigma}
        \end{array}\right] \geq 0.
        }
    \end{gathered}   
    $$
    \label{lem: QC for K mappings}
    \end{lemma}
    \begin{proof}
        The proof is given in Appendix \uppercase\expandafter{\romannumeral2}.
    \end{proof}
    
    Notice that as $K=1$, Lemma \ref{lem: QC for one mapping} is a case of Lemma \ref{lem: QC for K mappings}. In the relevant sections below, we would pre-state the value of the parameter $K$ and the input $x$ for NN mapping vector.
    
\section{Main Result}
\label{sec: Main Result}
    The approach in the previous section can be summarized as follows: by using isolation of non-linearity for {\color{black}an NN} mapping and QCs for activation functions, we extend the characterization of a single NN mapping to NN mapping vector composed of several NN mappings. In this section, we utilize this approach to answer the questions proposed in Section \ref{sec: problem formulation}. 
\subsection{Neural Observers for Systems without Uncertainty}
    For the intuitiveness and simplicity of the arguments, we start analyzing of the neural observability of linear systems. First, we formally state our main result for Problem \ref{pro: no for linear systems} in the following theorem.
    \begin{theorem}
   We consider {\color{black}an NN} mapping $\pi_{\theta}$ with $\theta$ and a vector $\lambda_{\sigma}$ that satisfies the quadratic constraint in Lemma \ref{lem: QC for one mapping}. We update $W^{L+2}$ and $W^1$ in \eqref{eq: R_pi and R_xi} to $W^{L+2}C$ and $W^1C$, respectively. If there exists a matrix $P\in\mathcal{L}(\mathbf{R}^{n_s})$ and $P\succ O$ such that 
    \begin{equation}
            R_{\pi}^{\top}
            \left[
                \begin{array}{cc}
                    A^{\top}P + PA & P\\
                    P&O
                \end{array}
            \right]R_{\pi}+
            R_{\xi}^{\top}\Psi_{\sigma}^{\top} M_{\sigma}(\lambda_{\sigma}) \Psi_{\sigma}R_{\xi} \prec O,
            \label{eq: LMI for LTI without control}
        \end{equation}
        then the LTI system \eqref{eq: linear systems} is {\color{black}neural exponentially observable}, equivalently, 
        $$
        \|x(t)-\widehat{x}(t)\|_2\leq M\exp^{-\kappa t}\{\|x(0)\|_2+\|\widehat{x}(0)\|_2\},
        $$
        where $R_{\pi}$, $R_{\xi}$, and $\Psi_{\sigma}$, $M_{\sigma}(\lambda_{\sigma})$ are defined in \eqref{eq: R_pi and R_xi}, and \eqref{eq: Psi_sigma and M_sigma}, respectively.
        \label{theo: LMI for LTI without control}
    \end{theorem}
    \begin{proof}
        {\color{black} The proof is provided in Appendix \uppercase\expandafter{\romannumeral3}.}
    \end{proof}
    
    Furthermore, it is not difficult to imply that the neural exponential observability of system \eqref{eq: linear systems} and the existence of $\theta$ in $\pi_{\theta}$ depend heavily on the existence of $P$, i.e., the solution of LMI \eqref{eq: LMI for LTI without control}. Hence, a natural sub-question is: under what conditions does the solution $P$ in LMI exist? To our best knowledge, this question has not been effectively solved in the NN-based closed-loop control (for example, \cite[Theorem 1]{yin2022stability} and \cite[Theorem 1]{pauli2021offset}) at present. Therefore, we present the following proposition to answer this sub-question.
    
    \begin{proposition}
        We set $\alpha_{\sigma}=\mathbf{0}_{n_{\sigma}}$, $\widetilde{A}=A+N_{\pi x}$ with $N_{\pi x}=W^{L+2}C$, $Q\in \mathcal{L}(\mathbf{R}^{n_s})$ is a diagonal matrix with positive diagonal entries, and $P=\int_0^{\infty}e^{\widetilde{A}^{\top}t}Qe^{\widetilde{A}t}\mathrm{d}t$. We suppose that there exists $\lambda_{\sigma} > \bm{0}_{n_{\sigma}}$ such that 
        \begin{itemize}
            \item[(\romannumeral1)] $\|M_1\|_{\infty}\leq \min_{i}(q_{i})$, where $M_1=-PN_{\pi w}-N_{\xi x}^{\top}R_1$ with $R_1=\operatorname{diag}(\lambda_{\sigma} \circ \beta_{\sigma})\footnote{``$\circ$" represents the Hadamard product}$, and $q_{i}$ is the $i_{\text{th}}$ diagonal entry in $Q$,
            \item[(\romannumeral2)] $\|M_1^{\top}\|_{\infty} + \|M_2\|_{\infty}\leq 2\min_{i}(\lambda_{\sigma,i})$, where $\lambda_{\sigma,i}$ is the $i_{\text{th}}$ element of $\lambda_{\sigma}$, $M_2 = R_1N_{\xi w}+N_{\xi w}^{\top}R_1$.
        \end{itemize}
        Then LMI (\ref{eq: LMI for LTI without control}) has a solution $P$ if and only if $(C,A)$ is observable.
        \label{prop: existence of solution of LMI}
    \end{proposition}
    \begin{proof}
        {\color{black} The proof is provided in Appendix \uppercase\expandafter{\romannumeral1}.}
    \end{proof}
    
    \begin{remark}
        It should be noticed that $W^{L+2}$ in the shortcut connection of the NN plays an essential role in the construction of the above solution $P$ by pole assignments. Moreover, from (\romannumeral1) and (\romannumeral2), the solution $P$ also be utilized to guarantee the existence of a reliable NN mapping $\pi_{\theta}$. We note that $(C, A)$ is observable if $A$ is Hurwitz. Then, we can take $\widetilde{A}=A$ by setting $W^{L+2}=O$ from the NN mapping $\pi_{\theta}$, which means that the residual neural network defined in \eqref{eq: resnets} will degenerate to a fully-connected NN. Therefore, Corollary \ref{coro: existence of solution of LMI} shows that the LMI \eqref{eq: LMI for LTI without control} solution exists in this case.
        \label{rm: role explanation for shortcut}
    \end{remark}
    
    \begin{corollary}
        Let the assumptions (\romannumeral1) and (\romannumeral2) be still satisfied. We set $\alpha_{\sigma}=\bm{0}_{n_{\sigma}}$ but $\widetilde{A}=A$, and corresponding matrices $Q,\ P$ defined in Proposition \ref{prop: existence of solution of LMI}. Then $P$ is a solutions for LMI \eqref{eq: LMI for LTI without control} if and only if $A$ is Hurwitz.
    \label{coro: existence of solution of LMI}
    \end{corollary}
    \begin{proof}
    The proof is a direct extension of Proposition \ref{prop: existence of solution of LMI}.
    \end{proof}
    
    The next theorem gives the necessary conditions to achieve the control target in Problem \ref{prop: control problem for linear systems} by using the measurement $\widehat{x}(t)$. From \eqref{eq: linear systems}-\eqref{eq: nos for linear systems}, by utilizing the NN controller of the form $u(t)=\pi_{\theta_2}(\widehat{x})$, it is not difficult to obtain
        \begin{equation}
            \left\{
                \begin{aligned}
                \dot{x}&=A x+B \pi_{\theta_1}(\widehat{x}) \\
                \dot{\widehat{x}}&=A \widehat{x}+B \pi_{\theta_1}(\widehat{x}) + \pi_{\theta_2}\left(C(\widehat{x}-x)\right)
                \end{aligned}
            \right.
            \label{eq: system and nos for LTI}
        \end{equation}
        By denoting $x_1 \triangleq  x$, $x_2\triangleq \widehat{x}-x$, and $\bm{x}^{\top}=[x_1^{\top},x_2^{\top}]$, equation \eqref{eq: system and nos for LTI} turns into
        $$
        \frac{\mathrm{d}\bm{x}}{\mathrm{d}t}=\widehat{A}\bm{x}+v(\bm{x}),
        $$
        where $\widehat{A}=\operatorname{diag}\left(A,A\right)$ and $v(\bm{x})=\left[\begin{array}{c}
            B\pi_{\theta_1}([I,I]\bm{x})\\
            \pi_{\theta_2}([O,C]\bm{x})
        \end{array}\right].$ Correspondingly, as $K=2$, we also treat $\bm{x}(t)$ as an input variable of the NN mapping vector $\bm{\pi_{\theta}}=[\pi^{\top}_{\theta_1},\pi^{\top}_{\theta_2}]^{\top}$.
    \begin{theorem}
       We consider two NN mappings $\pi_{\theta_1}$ and $\pi_{\theta_2}$ with parameters $\theta_i=\left(L_i,n_{\sigma_i}, W_i^{1},\cdots,W_i^{L_i+2}\right),\ i=1,2$. Let parameter $K$ in Lemma \ref{lem: QC for K mappings} be equal to $2$, and $T_1^1$, $T_1^2$, $T_2^1$,  $T_2^2$ in \eqref{eq: R_pi and R_xi 2} are equal to $[I,I]$, $B$, $[O,C]$, $I$, respectively. We suppose that there exists a matrix $\widehat{P}\in \mathcal{L}(\mathbf{R}^{2n_s})$ and $\widehat{P}\succ O$ such that
        \begin{equation}
        \begin{aligned}
            \widehat{R_{\bm{\pi}}} ^{\top}
            \left[
                \begin{array}{cc}
                    \widehat{A}^{\top}\widehat{P}  + \widehat{P}\widehat{A} & \widehat{P}\\
                    \widehat{P}&O
                \end{array}
            \right]&\widehat{R_{\bm{\pi}}}\\
            +&\widehat{R_{\xi}}^{\top}\mathbf{\Psi}(2)^{\top} \mathbf{M}(2) \mathbf{\Psi}(2)\widehat{R_{\xi}} \prec O,
        \end{aligned}
        \label{eq: LMI for LTI with control}
        \end{equation}
        where $\widehat{R_{\bm{\pi}}}$, $\widehat{R_{\xi}}$, and $\bm{\Psi}(2)$, $\bm{M}(2)$ are defined in \eqref{eq: R_pi and R_xi 2}, and \eqref{eq: Psi_sigma and M_sigma 2} as $K=2$, respectively.
        Then, the LTI system \eqref{eq: linear systems} is {\color{black}neural exponentially observable} and globally exponentially stable.
        \label{theo: LMI for LTI with control}
    \end{theorem}
    \begin{proof}
        {\color{black} The proof is provided in Appendix \uppercase\expandafter{\romannumeral3}.}
    \end{proof}
    
    Reasonably, the existence of solutions to LMI \eqref{eq: LMI for LTI with control} needs to be taken into account, and below, we propose one class of solutions satisfying LMI \eqref{eq: LMI for LTI with control}.
    \begin{proposition}
        By setting $\alpha_{\sigma_i}=\mathbf{0}_{n_{\sigma_i}}$, 
        \begin{itemize}
            \item $\widetilde{A}_1=A+BW_1^{L_1+2}$, $P_1=\int_0^{\infty}e^{\widetilde{A}_1^{\top}t}Q_1e^{\widetilde{A}_1t}\mathrm{d}t$ with a diagonal matrix $Q_1\succ O$,
            \item and $\widetilde{A}_2=A+W_2^{L_2+2}C$, $P_2=\int_0^{\infty}e^{\widetilde{A}_2^{\top}t}Q_2e^{\widetilde{A}_2t}\mathrm{d}t$ with a diagonal matrix $Q_2\succ O$,
        \end{itemize}
        we assume that there exists $\lambda_{\sigma}^2>\bm{0}_{n^2_{\sigma}}$ such that
        \begin{itemize}
            \item[(\romannumeral3)] $\|M_1\|_{\infty}+\|M_3\|_{\infty}\leq \min_i(\widehat{q}_i)$, where $M_1=-\widehat{P}\widehat{N_{\bm{\pi} w}}-\widehat{N_{\xi x}}^{\top}R_1$ with $R_1=\operatorname{diag}(\lambda_{\sigma}^2\circ \beta_{\sigma}^2)$, $\widehat{P}=\operatorname{diag}(P_1,P_2)$, $q_i$ is the $i_{\text{th}}$ diagonal entry in $\widehat{Q}=\operatorname{diag}(Q_1,Q_2)$, and $M_3=\left[\begin{array}{cc}
                O & P_1BW_1^{L_1+2}\\
                \star & O
            \end{array}\right]$,
            \item[(\romannumeral4)] $\|M_1^{\top}\|_{\infty} + \|M_2\|_{\infty}\leq 2\min_{i}(\lambda_{\sigma,i})$, where $\lambda_{\sigma,i}$ is the $i_{\text{th}}$ element of $\lambda_{\sigma}^2$, $M_2 = R_1\widehat{N_{\xi w}}+\widehat{N_{\xi w}}^{\top}R_1$.
        \end{itemize}  
        Then LMI \eqref{eq: LMI for LTI with control} has a solution $\widehat{P}$ if and only if $(C, A)$ is observable and $(A,B)$ is controllable.
        \label{prop: existence of solution of LMI 2}
    \end{proposition}
    \begin{proof}
        {\color{black} The proof is provided in Appendix \uppercase\expandafter{\romannumeral1}.}
    \end{proof}
    
    In the following sections, to make the content as concise as possible, we would not prove the existence of LMI (\eqref{eq: LMI for integral-chain systems}, \eqref{eq: LMI for integral-chain systems 2} and \eqref{eq: LMI for non-integral chain systems}) solutions in detail once the conditions of observability or stabilization are satisfied since the proofs are similar to Proposition \ref{prop: existence of solution of LMI}-\ref{prop: existence of solution of LMI 2}.
\subsection{Uncertainty is Effectively Dealt by Neural Observers}
    We construct an extended state, $$x_{n+1}(t)=\mathcal{F}(t,\bm{x},w),$$ for \eqref{eq: integral-chain nonlinear systems_2} and then redefine system \eqref{eq: integral-chain nonlinear systems_2} as follows
    \begin{equation}
        \left\{
        \begin{aligned}
            \dot{\bm{\widetilde{x}}}&=\widetilde{\mathcal{A}}\bm{\widetilde{x}}+\widetilde{L}(t,\bm{x},w,u),\\
            y&=\widetilde{c}\bm{\widetilde{x}},
        \end{aligned}
        \right.
        \label{eq: integral-chain nonlinear systems_3}
    \end{equation}
    where $\bm{\widetilde{x}} = [\bm{x}^{\top},x_{n+1}]^{\top}$, $\widetilde{\mathcal{A}} = (a_{ij})_{(n+1)\times (n+1)}$ is defined by 
    $$
        a_{ij}=\left\{\begin{array}{l}
        1,\ i+1=j  \\
        0,\ \text{else}
        \end{array}\right.,\ \widetilde{c}=[1,0, \cdots, 0]\in \mathcal{L}(\mathbf{R}^{n+1},\mathbf{R}), 
    $$
    
    $$
        \widetilde{L}(t,\bm{x},w,u)=[0, \cdots, bu(t),\frac{\mathrm{d}}{\mathrm{d}t}\mathcal{F}(t,\bm{x},w)]^{\top}. 
    $$
    Correspondingly, the output of neural observer \eqref{eq: nos for integral-chain nonlinear systems} is redefined as $\widehat{y}=\widetilde{c}[\widehat{\bm{x}}^{\top},\widehat{x}_{n+1}]^{\top}$. We suppose that $\mathcal{F}$ satisfies Assumption \ref{ap: the uncertainty for integral-chain} and denote $h(t)\triangleq\frac{\mathrm{d}}{\mathrm{d}t}\mathcal{F}(t,\bm{x},w)$, $\eta_i\triangleq\epsilon^{-n-1+i}(x_i-\widehat{x}_i),i=1,\cdots,n+1$. Then, from \eqref{eq: integral-chain nonlinear systems_3} and \eqref{eq: nos for integral-chain nonlinear systems} we derive an error system for \eqref{eq: integral-chain nonlinear systems} and \eqref{eq: nos for integral-chain nonlinear systems}:
    \begin{equation}
     \left\{
        \begin{aligned}
            \epsilon\dot{\eta}_i&=\eta_{i+1}-\pi_{\theta_i}(\widetilde{c}\bm{\eta}),\quad i=1,2,\cdots,n, \\
            \epsilon\dot{\eta}_{n+1}&=-\pi_{\theta_{n+1}}(\widetilde{c}\bm{\eta})+\epsilon h(t),
        \end{aligned}  
    \right.   
    \label{eq: error system for integral-chain systems}
    \end{equation}
    where $\bm{\eta}=[\eta_1,\cdots,\eta_{n+1}]^{\top}$. From the above error system \eqref{eq: error system for integral-chain systems}, as $K=n+1$ in Lemma \ref{lem: QC for K mappings}, we treat $\bm{\eta}$ as the input of NN mapping vector $\bm{\pi_{\theta}}=[\pi_{\theta_1},\cdots,\pi_{\theta_{n+1}}]^{\top}$ by setting $T_1^i=\widetilde{c},T_2^i=1,\ i=1,\cdots,n+1$. Now, we formally propose the result for Problem \ref{pro: dc for integral-chain systems}.
    \begin{theorem}
     We consider $n+1$ NN mappings $\pi_{\theta_i}$ with parameters $\theta_i=\left(L_i,n_{\sigma_i}, W_i^{1},\cdots,W_i^{L_i+2}\right),\ i=1,\cdots,n+1$. We assume that
    \begin{itemize}
        \item[(1)] $K=n+1$ in Lemma \ref{lem: QC for K mappings}, and $T_i^1=\widetilde{c},T_i^2=1,\ i=1,\cdots,n+1$ in \eqref{eq: R_pi and R_xi 2},
        \item[(2)] there exists a matrix $P\in \mathcal{L}(\mathbf{R}^{n+1})$ and $P\succ O$ such that
        \begin{equation}
            \begin{aligned}
                \widehat{R_{\bm{\pi}}}^{\top}&
                \left[
                    \begin{array}{cc}
                        \widetilde{\mathcal{A}}^{\top}P + P\widetilde{\mathcal{A}} & -P\\
                        -P&O
                    \end{array}
                \right]\widehat{R_{\bm{\pi}}} \\
                +&\widehat{R_{\xi}}^{\top}\mathbf{\Psi}(n+1)^{\top} \mathbf{M}(n+1) \mathbf{\Psi}(n+1)\widehat{R_{\xi}} \prec O.
            \end{aligned}
        \label{eq: LMI for integral-chain systems}
        \end{equation}
    \end{itemize}
    Then we have the following results:
    \begin{itemize}
            \item[(\textbf{\uppercase\expandafter{\romannumeral1}})] Neural observability: for all $x(0)\in \mathbb{R}^n$, $\|\bm{x}(t)-\widehat{\bm{x}}(t)\|_2\rightarrow 0$ as $\epsilon\rightarrow 0^+$ for $t \in \mathcal{T}_{\geq T}$ with $T>0$.
            \item[(\textbf{\uppercase\expandafter{\romannumeral2}})] Total uncertainty $x_{n+1}=F(t,\bm{x},w)$ can be measured by $\widehat{x}_{n+1}$ as $\epsilon\rightarrow 0^+$, i.e., $\lim_{\epsilon\rightarrow 0^+}|x_{n+1}-\widehat{x}_{n+1}|=0$.
        \end{itemize}
        \label{theo: LMI for integral-chain systems}
    \end{theorem}
    \begin{proof}
        {\color{black} The proof is provided in Appendix \uppercase\expandafter{\romannumeral3}.}
    \end{proof}
    \begin{remark}
        For \eqref{eq: LMI for integral-chain systems}, due to the observability of $(\widetilde{c},\widetilde{\mathcal{A}})$, we can select $W^{L+2}=[W_1^{L_{1}+2},\cdots,W_{n+1}^{L_{n+1}+2}]^{\top}$ as a pole assignment matrix such that $\mathcal{A}+\widehat{R_{\bm{\pi} x}}=\mathcal{A}+W^{L+2}\widetilde{c}$ is Hurwitz, where $\widehat{R_{\bm{\pi} x}}$ is a block matrix of $\widehat{R_{\bm{\pi}}}$. Then, by a similar analysis with Proposition \ref{prop: existence of solution of LMI}-\ref{prop: existence of solution of LMI 2}, it is not difficult to check that LMI \eqref{eq: LMI for integral-chain systems} has solutions under some given conditions.
        \label{rem: existence of solution of LMI 3}
    \end{remark}
    
        To decrease the computational complexity for LMI \eqref{eq: LMI for integral-chain systems} and avoid the consequences of sparsity \cite{zhang2018efficient,madani2017finding}, we can take that the NN mappings in \eqref{eq: nos for integral-chain nonlinear systems} are identical, i.e., $\pi_{\theta_i}(\cdot)=\pi_{\theta}(\cdot),i=1,\cdots,n+1$. Moreover, the gains of the NN mapping are equal to $\epsilon^{n+1-i}b_i$, i.e.,
        \begin{equation*}
        \left\{
        \begin{aligned}
            \dot{\widehat{x}}_i&=\widehat{x} _{i+1}+\epsilon^{n-i}\pi_{\theta}(\epsilon^{-n}(y-\widehat{y})),\widehat{x}_i(0)=\widehat{x}_{i,0},\\ \qquad &\qquad\qquad\qquad\qquad\qquad\qquad\qquad i=1,\cdots,n-1, \\
            \dot{\widehat{x}}_n&=\widehat{x} _{n+1}+\pi_{\theta}(\epsilon^{-n}(y-\widehat{y}))+bu,\widehat{x}_n(0)=\widehat{x}_{n,0},\\
            \dot{\widehat{x}}_{n+1}&=\epsilon^{-1}\pi_{\theta}(\epsilon^{-n}(y-\widehat{y})),\widehat{x}_{n+1}(0)=\widehat{x}_{n+1,0},\\
            \widehat{y}&=c\widehat{\bm{x}},
        \end{aligned}  
        \right.
        \end{equation*}
    Then we present the following corollary to solve the sparsity of LMI \eqref{eq: LMI for integral-chain systems}.
    \begin{corollary}
        We consider this NN map $\pi_{\theta}$ with one parameter $\theta=\left(L,n_{\sigma},W^{1},\cdots,W^{L+2}\right)$ and re-assume
        \begin{itemize}
            \item[(1)] $K=1$ in Lemma \ref{lem: QC for K mappings}, and $T_1^1=\widetilde{c},T_1^2=1$ in \eqref{eq: R_pi and R_xi 2},
            \item[(2)] Let $b$ be the vector $[b_1,\cdots,b_{n+1}]^{\top}$. We suppose that there exists a matrix $P\in \mathcal{L}(\mathbf{R}^{n+1})$ and $P\succ O$ such that
            \begin{equation}
                \begin{aligned}
                    \widehat{R_{\bm{\pi}}}^{\top}
                    \left[
                        \begin{array}{cc}
                            \widetilde{\mathcal{A}^{\top}}P + P\widetilde{\mathcal{A}} & -Pb\\
                            -Pb&O
                        \end{array}
                    \right]&\widehat{R_{\bm{\pi}}} \\
                    +\widehat{R_{\xi}}^{\top}\mathbf{\Psi}(1)^{\top} \mathbf{M}(1)& \mathbf{\Psi}(1)\widehat{R_{\xi}} \prec O,
                \end{aligned}
            \label{eq: LMI for integral-chain systems 2}
            \end{equation}
            where $\widehat{R_{\bm{\pi}}}$, $\widehat{R_{\xi}}$, $\mathbf{\Psi}(1)$ and $\mathbf{M}(1)$ are defined in \eqref{eq: R_pi and R_xi 2} and \eqref{eq: Psi_sigma and M_sigma 2}.
        \end{itemize}
        Then we can still obtain three results in Theorem \ref{theo: LMI for integral-chain systems}, including neural observability, and the measurement of the total uncertainty $\mathcal{F}(t,\bm{x},w)$.
        \label{coro: non-sparsity solution for LMI}
    \end{corollary}
    \begin{proof}
            By directly extending the proof of Theorem \ref{theo: LMI for integral-chain systems}, the proof of Corollary \ref{coro: non-sparsity solution for LMI} can be obtained trivially.
    \end{proof}
    
    \subsection{General Uncertainty in Linear Dynamics Can be Dealt by Neural Observers}
    Before showing Theorem \ref{theo: LMI for nonintegral-chain systems} for Problem \ref{pro: d for nonintegral-chain systems}, we introduce a necessary lemma. Furthermore, finally, we present the last Theorem for systems \eqref{eq: nonintegral-chain systems}.
    \begin{lemma}
    $(\mathbf{C},\mathbf{A}_{\epsilon})$ is observable if $A$, $C$, $B_w$ satisfy the \textit{extending observable condition} defined in Assumption \ref{ap: the uncertainty for nonintegral-chain systems},
    where $\mathbf{C}$ is defined in Assumption \ref{ap: the uncertainty for nonintegral-chain systems} (1), and $\mathbf{A}_{\epsilon}=\left[\begin{array}{cc}
        \epsilon A & B_w \\
        O & O
    \end{array}\right]$ with $\epsilon>0$.
    \label{lem: Positive definite under perturbation}
    \end{lemma}
    \begin{proof}
        The proof is given in Appendix \uppercase\expandafter{\romannumeral2}.
    \end{proof}
    \begin{theorem}
    We consider two NN mappings $\pi_{\theta_1}$ and $\pi_{\theta_2}$ with parameter $\theta_i=\left(L_i,n_{\sigma_i}, W_i^{1},\cdots,W_i^{L_i+2}\right),\ i=1,2$. Let parameter $K$ in Lemma \ref{lem: QC for K mappings} be equal to $2$, and $T_i^1$, $T_i^2$ in \eqref{eq: R_pi and R_xi 2} are equal to $\mathbf{C}$, $I$, respectively. We suppose that there exists a positive definite matrix $\mathbf{P}\in\mathcal{L}(\mathbf{R}^{n_s+n_q})$  such that
    \begin{equation}
        \begin{array}{cc}
            \begin{aligned}
                D_1 \triangleq 
                \widehat{R_{\bm{\pi}}}^{\top}
                \left[
                    \begin{array}{cc}
                        \mathbf{A}_{\epsilon}^{\top}\mathbf{P} +\mathbf{P}\mathbf{A}_{\epsilon} & -\mathbf{P}\\
                        -\mathbf{P}&O
                    \end{array}
                \right]&\widehat{R_{\bm{\pi}}} +\\
                \widehat{R_{\xi}}^{\top}\mathbf{\Psi}(2)^{\top}& \mathbf{M}(2) \mathbf{\Psi}(2)\widehat{R_{\xi}} \prec O.
            \end{aligned}\\
        \end{array}
        \label{eq: LMI for non-integral chain systems}
    \end{equation}
   
    Under Assumption \ref{ap: the uncertainty for nonintegral-chain systems} aforementioned, then for $\epsilon >0$, the system \eqref{eq: nonintegral-chain systems} is neural observable in the following sense:
        \begin{itemize}
            \item[$\bullet$]$\lim_{\epsilon \to 0^+} \|x_i(t)-\hat{x}_i(t)\|_2=0$ for all $t\in\mathcal{T}_{\geq a},\ a>0.$
            \item[$\bullet$] $\overline{\lim} _{t \to \infty}\|x_i(t)-\hat{x}_i(t)\|_2\leq O(\epsilon^{2-i})$.
        \end{itemize}
    \label{theo: LMI for nonintegral-chain systems}
    \end{theorem}
    \begin{proof}
        {\color{black} The proof is provided in Appendix \uppercase\expandafter{\romannumeral3}.}
    \end{proof}
    \textcolor{black}{
    \begin{remark}
        For \eqref{eq: LMI for non-integral chain systems}, since $(\mathbf{C}, \mathbf{A}_{\epsilon})$ is observable from Lemma \ref{lem: Positive definite under perturbation}, we can construct $W^{L+2}=\left[(W^{L_1+2}_1)^{\top},(W^{L_2+2}_2)^{\top}\right]^{\top}$ such that $\mathbf{A}_{\epsilon}+\widehat{R_{\bm{\pi}x}}=\mathbf{A}_{\epsilon}+W^{L+2}\mathbf{C}$ is Hurwitz. Subsequently, one can verify the existence of LMI \eqref{eq: LMI for non-integral chain systems} solutions $\mathbf{P}$ via a similar process with Proposition \ref{prop: existence of solution of LMI}-\ref{prop: existence of solution of LMI 2}.
    \label{rem: existence of solution of LMI 4}
    \end{remark}}

\section{Numerical Experiments}
\label{sec: Numerical Experiments}
    We apply the neural observers for three different dynamical models to demonstrate the effectiveness of our proposed analyses. In these examples, the LMIs \eqref{eq: LMI for LTI with control}, \eqref{eq: LMI for integral-chain systems}, and \eqref{eq: LMI for non-integral chain systems} are solved using the LMI Toolbox in MATLAB R2021a.
    \subsection{Linearized Aerodynamic Models of the X-29A Aircraft}
    We implement the neural control framework combining the neural observer \eqref{eq: nos for linear systems} and the NN controller \eqref{eq: controller for linear systems} to the X-29A aircraft, which is formulated in the following state-space form:
    \begin{equation*}
	\left\{
	\begin{aligned}
	\dot{x} =&Ax+Bu+w,\ x(0)\in\mathbf{R}^4,\\
	y = & Cx+v,
	\end{aligned}
	\right.
	\end{equation*}
    where the nominal system matrices $A$, $B$, and $C$ satisfying Assumption \ref{ap: control-observe condition} can be obtained from Table 9 in \cite{bosworth1992linearized}; $w\sim N(0,\frac{1}{10}I)$ and $v\sim N(0,\frac{1}{10}I)$ are process noises. The NNs $\pi_{\theta_i},i=1,2$ in \eqref{eq: nos for linear systems} and \eqref{eq: controller for linear systems} are both parameterized by three hidden layers ($n_1=n_2=n_3=3$) with $\operatorname{ReLU}$ / $\operatorname{tanh}$ as the activation function for all layers. 
    We {\color{black} further perform a comparison between neural observers with different activation functions ($\hat{x}_{i,R}$ and $\hat{x}_{i,T}$, denoting $\operatorname{ReLu}$ and $\operatorname{tanh}$ activations, respectively)} and {\color{black}the Kalman filter ($\hat{x}_{i,K}$)}, where $i=1,\cdots,4$. {\color{black}All initial values are set to be $\hat{x}(0)=0$.}
    \begin{figure}[t]
        \centering
        \includegraphics[scale=.25]{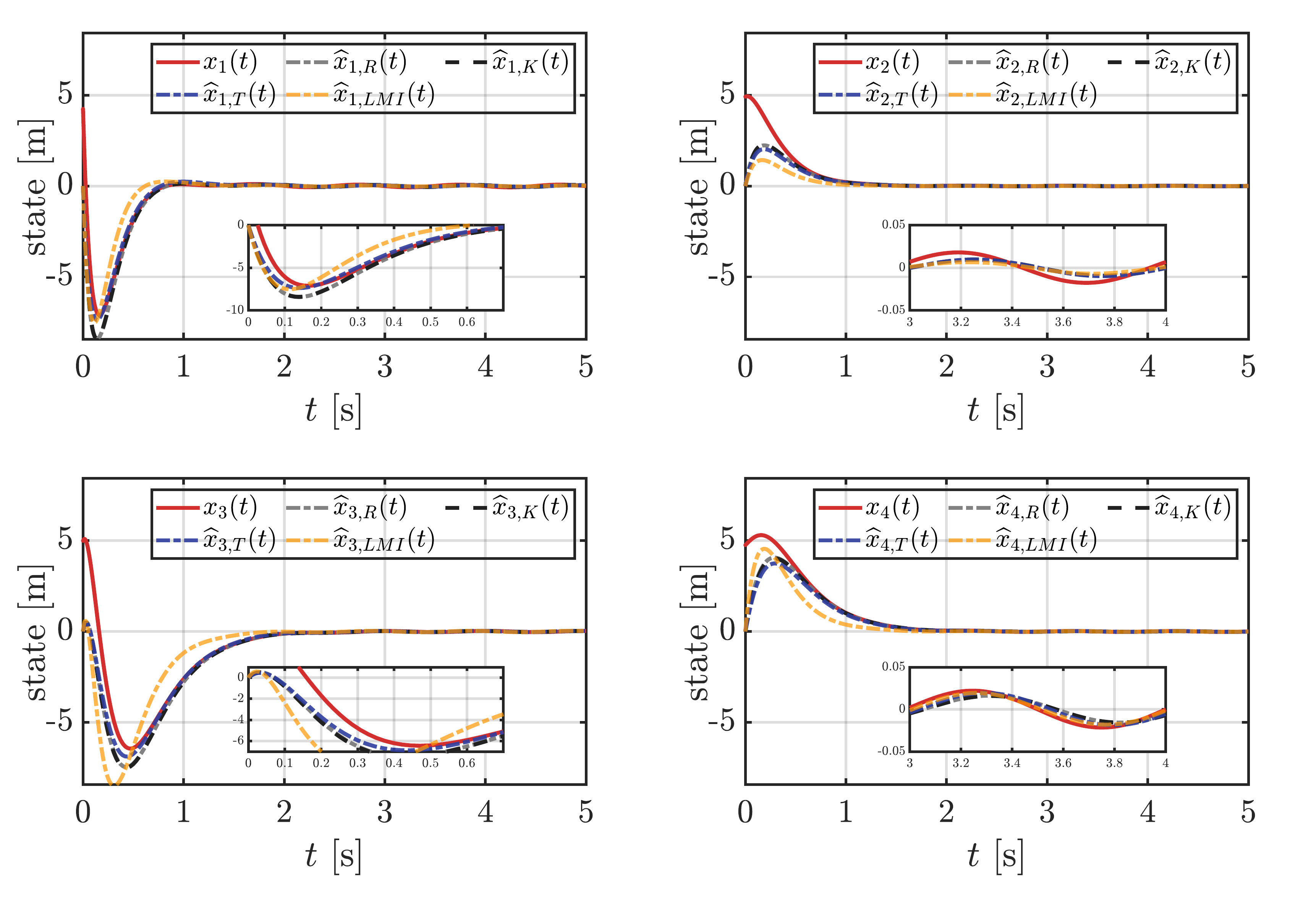}
        \caption{The state trajectories of X-29 aircraft under the neural control framework. \textbf{Red solid  line} depicts the state of the nominal system. \textbf{Grey and blue dotted lines} depict the {\color{black}estimated state} of neural observers equipped with $\operatorname{ReLu}$ and $\operatorname{tanh}$ as activation functions, respectively. \textbf{Black dotted line} represents the {\color{black}estimated state} based on the Kalman filter. \textbf{Orange dotted line} represents the {\color{black}estimated state} generated from a neural observer \eqref{eq: nos for linear systems} that dissatisfies the LMI \eqref{eq: LMI for LTI with control}.}
        \label{fig: X-29A aircraft}
    \end{figure}
    The system response and the output of the neural observer are depicted in Fig. \ref{fig: X-29A aircraft}. It is shown that $x_i(t),i=1,\cdots,4$ all converge to a tiny neighbourhood of $0$ and are well estimated by $\widehat{x}_{i,R}(t), \widehat{x}_{i,T}(t), \widehat{x}_{i,K}(t), i=1,\cdots,4$. In addition, the different choices of activation functions in neural observers only have slight impact on the observation in this scene. {\color{black} It is worth mentioning that the {\color{black}state} can be also estimated by $\widehat{x}_{i,LMI}(t), i=1,\cdots,4$, which are generated from a neural observer \eqref{eq: nos for linear systems} with $\hat{x}(0)=\bm{0}_4$ that dissatisfies the LMI \eqref{eq: LMI for LTI with control}, indicating that the LMI criterion \eqref{eq: LMI for LTI with control} for neural observers is overly conservative.}

    \subsection{A Second-order Nonlinear Model of the Inverted Pendulum}
    Next, to show the effectiveness of neural observers for \textcolor{black}{integrator chain} nonlinear systems, we consider the control of the nonlinear inverted pendulum system formulated by $\ddot{\theta}(t)=\frac{m g l \sin (\theta(t))-\varsigma_0 \dot{\theta}(t)+u(t)+w(t)}{m l^{2}}$, where $\theta(t)$ is the angular position (rad), and $w(t)=\sum_{i=1}^p a_i \sin(b_i t+\phi_i)$ is the external disturbance. By denoting $x_1=\theta$ and $x_2=\dot{\theta}$, we rewrite state-space form for inverted pendulum system
    \begin{equation*}
	\left\{
	\begin{aligned}
	\dot{x_1} =&x_2,\ x_1(0)=x_{1,0},\\
	\dot{x_2}=&\frac{m g l \sin (x_1)-\varsigma_0 x_2+u(t)+w(t)}{m l^{2}},\ x_2(0)=x_{2,0},\\
	y =& x_1,
	\end{aligned}
	\right.
	\end{equation*}
	 where $m$, $l$, $\varsigma_0$ represent the mass (kg), the length (m), and the friction coefficient (Nms/rad), respectively. However, $m=m_0+\delta_m m_0$ and $l=l_0+\delta_l l_0$ are the uncertain parameters, where $m_0,l_0$ denote the nominal value and $\delta_m,\delta_l$ are parameter perturbation coefficients sketching the uncertainty of parameters. Without loss of generality, we consider that $m_0=1$, $l_0=1$, $|\delta_l|,|\delta_m|\in [0,0.1]$, $\varsigma_0=0.5$, and $w(t)=0.1\sin(4\pi t)+0.2\cos(2\pi t)+0.2\sin(3\pi t-\pi/7)$. The following neural observer is designed without involving the parameters $m$, $l$ and $\varsigma_0$:
	\begin{equation*}
	\left\{
	\begin{aligned}
	\dot{\widehat{x}}_1&=\widehat{x} _{2}+\epsilon^{1}\pi_{\theta_1}(\epsilon^{-2}(y-\widehat{y})),\ {\color{black}\widehat{x}_1(0)=0,}\\
    \dot{\widehat{x}}_2&=\widehat{x} _{3}+\pi_{\theta_2}(\epsilon^{-2}(y-\widehat{y}))+bu,\ {\color{black}\widehat{x}_2(0)=0,} \\
    \dot{\widehat{x}}_{3}&=\epsilon^{-1}\pi_{\theta_{3}}(\epsilon^{-2}(y-\widehat{y})),\ {\color{black}\widehat{x}_3(0)=0,}\\
	\widehat{y}& = \widehat{x}_1,
	\end{aligned}
	\right.
	\end{equation*}
	and the feedback control law $u(t)$ (Nm) is designed by $u(t)=\rho \sum_{i=1}^2 k_i \operatorname{sat}_{M_i}(\rho^{n-i}\widehat{x}_i(t))-\operatorname{sat}_{M_{3}}(\widehat{x}_{3}(t))$.
	
    In the corresponding neural observer \eqref{eq: nos for integral-chain nonlinear systems}, we design that (\romannumeral1) the gain $\epsilon=0.1$; (\romannumeral2) the NN $\pi_{\theta_i},i=1,\cdots,3$ are all parameterized by two hidden layers ($n_1=3$ and $n_2 = 2$) with $\tanh$ as the activation function for all layers. As for the control law, we set $\rho=1$, $k_1=-25,k_2=-10$, and $M_1=M_2=M_3=10$ (More details about the parameters setting can be seen in \cite{zhao2016active}). {\color{black} We also compare the above neural observer with the gain scheduled Luenberger observers (GSLO) \cite{benavides2014gain}, which is designed by involving $m_0,l_0$ and selecting $\delta_m=\delta_l=0.1$ and zero initial value.}
    \begin{figure}[t]
        \centering
        \includegraphics[scale=.3]{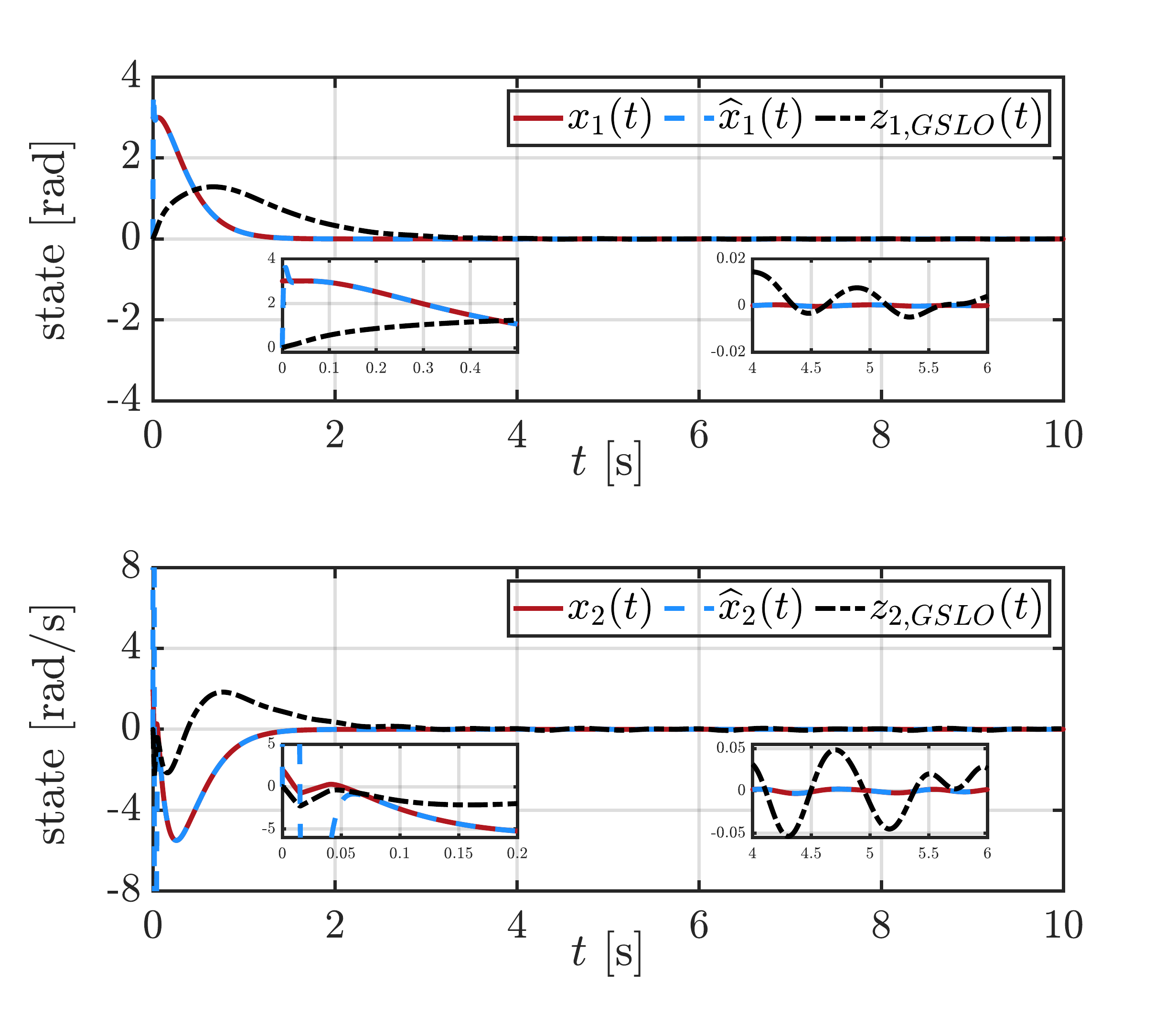}
        \caption{The state trajectories of the inverted pendulum. \textbf{Red solid  line} depicts the state $x_i,i=1,2$ of the nominal system. \textbf{Blue dotted lines} depicts the {\color{black}estimated state} $\widehat{x}_i,i=1,2$ of the neural observer. \textbf{Black chain line} represents the {\color{black}estimated state} $z_{i,GSLO},i=1,2$ based on the gain scheduled Luenberger observer~\cite{benavides2014gain}.}
        \label{fig: inverted pendulum}
    \end{figure}
    
    \begin{figure}[t]
        \centering
        \includegraphics[scale=.35]{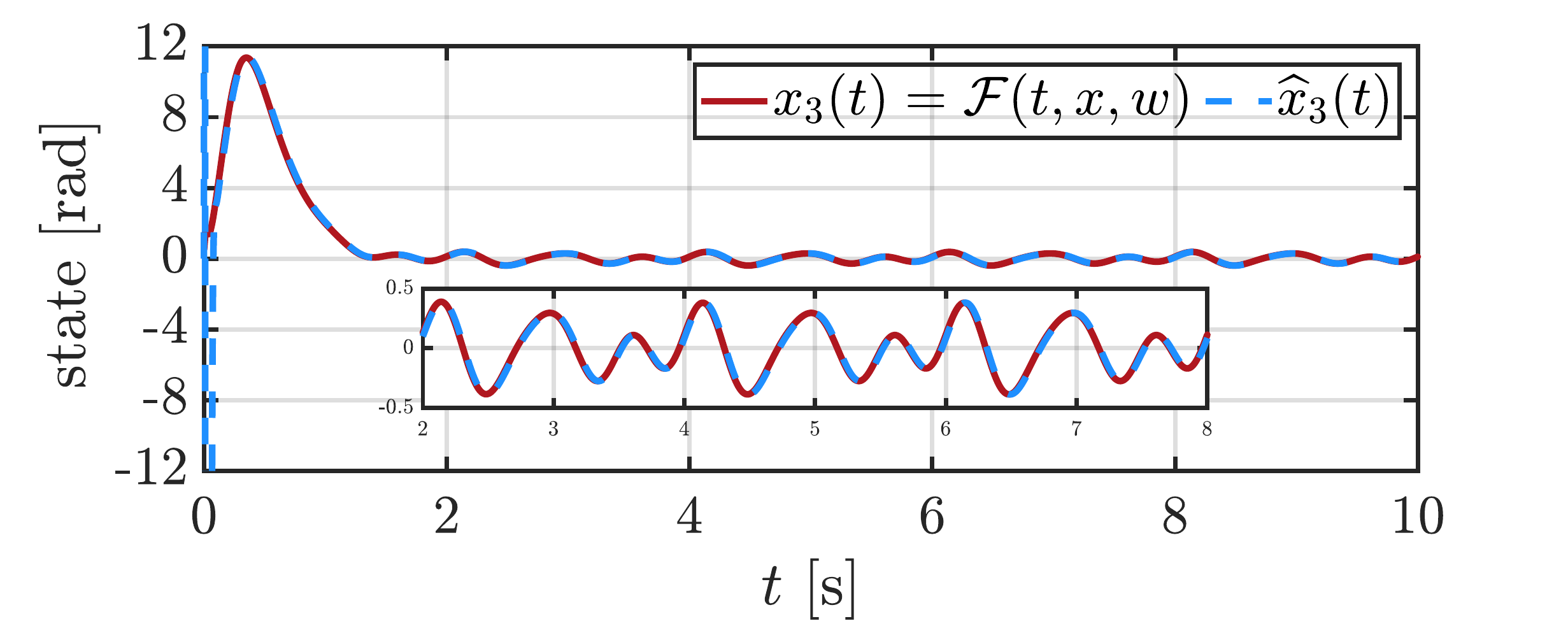}
        \caption{The trajectory of extended state $x_3=\frac{m g l \sin (x_1)-\varsigma_0 x_2+w(t)}{m l^{2}}$ of the inverted pendulum. The red and blue dotted lines depict the truth and the estimation from neural observer, respectively. }
        \label{fig: inverted pendulum uncertainty}
    \end{figure}
    
    As illustrated in Figs.~\ref{fig: inverted pendulum}-\ref{fig: inverted pendulum uncertainty}, we can conclude that the neural observer is not only more effective than the GSLO in tracking the state $x_1, x_2$, but also can estimate the extended state $x_3$ (total disturbance), which is not possible for GSLO.
    
    \subsection{The Dynamics of the Four-wheel Steering Vehicle}
    
    Finally, we implement the proposed neural observers \eqref{eq: nos for nonintegral-chain linear systems} to the four-wheel steering vehicle, which is modeled as a linear dynamic with a general uncertainty~\cite{alleyne1997comparison}:
    \begin{equation*}
        \left\{
        \begin{aligned}
            \dot{x}&=Ax+Bu+B_w\mathcal{K}(x,w,t),\ x(0)=x_0,\\
            y&=Cx,
        \end{aligned}
        \right.
    \end{equation*}
    where $x=[e,\dot{e},\Delta \psi, \dot{\Delta \psi}]$, with $(e,\Delta \psi)$ are defined as the perpendicular distance to the lane edge and the angle between the tangent to the straight section of the road; 
    $C=I$; 
    $A$, $B$, $B_w$, and $\mathcal{K}(x,w,t)$ are defined as follows:
    $$
    A = \left[\begin{array}{cccc}
    0 & 1 & 0 & 0 \\
    0 & \frac{C_{\alpha f}+C_{\alpha r}}{m U} & -\frac{C_{\alpha f}+C_{\alpha r}}{m} & \frac{a C_{\alpha f}-b C_{\alpha r}}{m U} \\
    0 & 0 & 0 & 1 \\
    0 & \frac{a C_{\alpha f}-b C_{\alpha r}}{I_{z} U} & -\frac{a C_{\alpha f}-b C_{\alpha r}}{I_{z}} & \frac{a^{2} C_{\alpha f}+b^{2} C_{\alpha r}}{I_{z} U}
    \end{array}\right],
    $$
    $$
    B = \left[\begin{array}{cc}
    0 & 0\\
    -\frac{C_{\alpha f}}{m} & -\frac{C_{\alpha r}}{m}\\
    0 & 0 \\
    -\frac{a C_{\alpha f}}{I_{z}} & \frac{b C_{\alpha f}}{I_{z}}
    \end{array}\right],\ B_w=\left[\begin{array}{cc}
    0 & 0\\
    1 & 0\\
    0 & 0\\
    0 & 1
    \end{array}\right].
    $$
    $$
    \mathcal{K}(\cdot)=\left[
    \begin{array}{c}
    0.1\sin(4t)+0.3\cos(2\pi t)+\frac{a C_{\alpha f}-b C_{\alpha r}-mU^{2}}{m\rho_c}\\
    0.2\cos(5t)+0.1\cos(6\pi t)+\frac{a^{2} C_{\alpha f}+b^{2} C_{\alpha r}}{I_{z}\rho_c}
    \end{array}\right].
    $$
    For simplicity, we denote that $x=[x_1,\cdots,x_4]^{\top}$. The parameters $C_{\alpha f},C_{\alpha r}, m,U,I_z,a,b$ represent the front cornering stiffness (N/rad), rear cornering stiffness (N/rad), mass (kg), longitudinal velocity (m/s), the moment of inertia $\operatorname{(kg/m^2)}$, distances from vehicle center of gravity to the front axle and rear axle, respectively, which are chosen to the nominal values obtained from Appendix A in \cite{alleyne1997comparison}. The constant road curvature $\rho_c$ in $\mathcal{K}(\cdot)$ can be chosen to be 400 (meters). 

    \begin{figure}[ht]
        \centering
        \includegraphics[scale=.25]{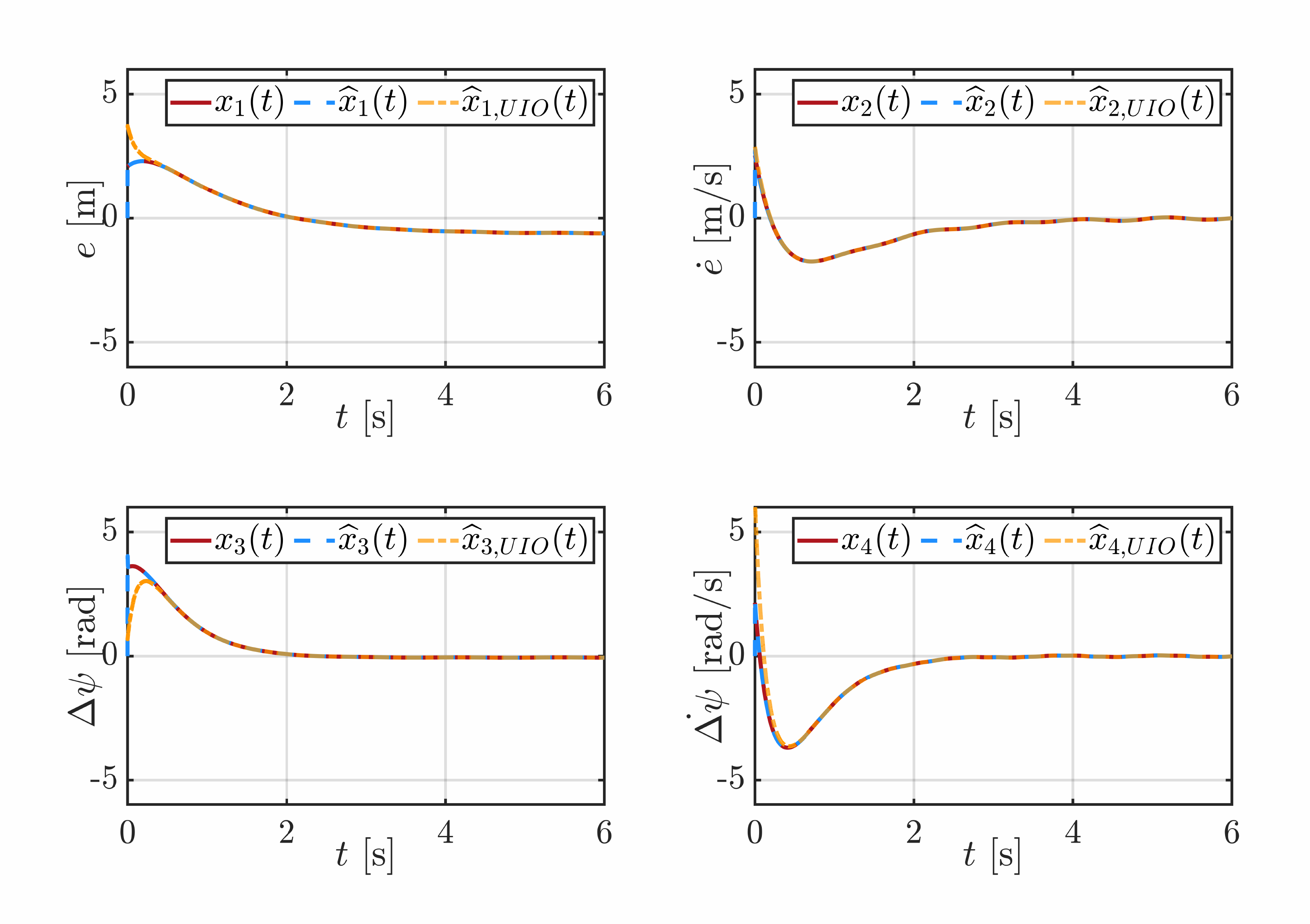}
        \caption{The trajectory of state $x_i(t),i=1,\cdots,4$ of the four-wheel steering wehicle under the neural observer. \textbf{Red solid line} represents the state of the nominal system. \textbf{Blue dotted lines} and \textbf{Orange dotted line} represent the {\color{black}estimated state} of the neural observer and the UIO, respectively. }
        \label{fig: Four-wheel Steering Vehicle}
    \end{figure}
    \begin{figure}[ht]
        \centering
        \includegraphics[scale=.25]{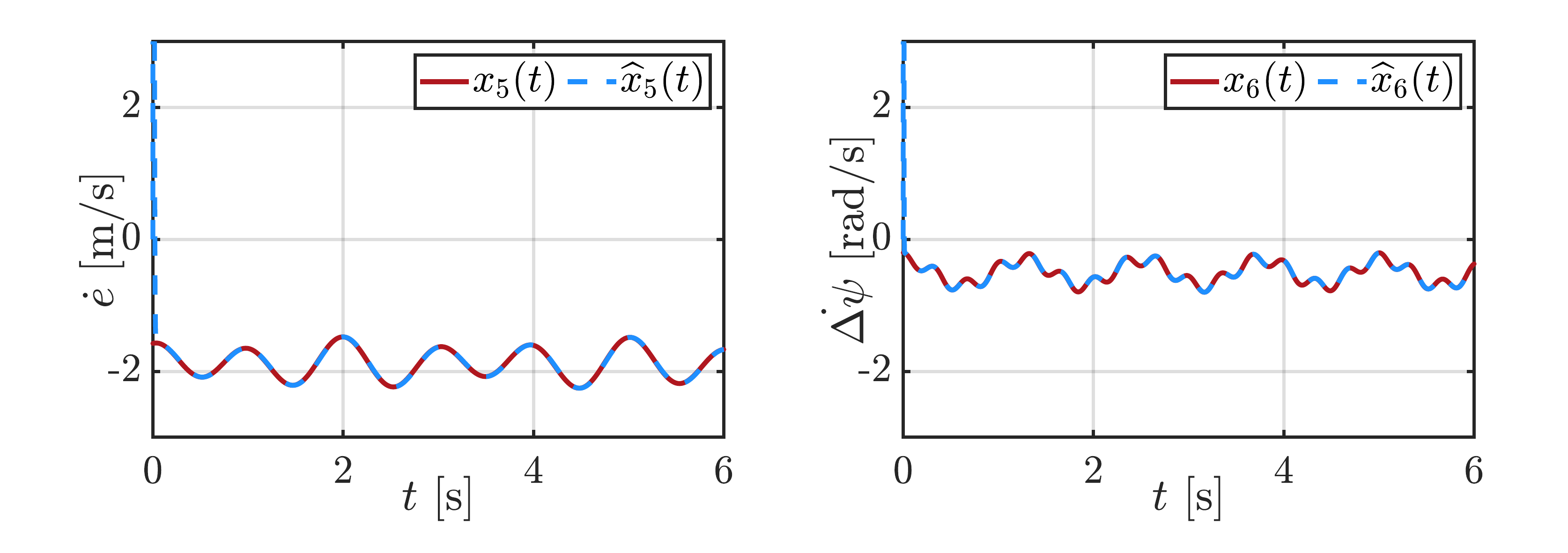}
        \caption{The estimation of the extended state $\mathcal{K}(x,w,t)\triangleq[x_5,x_6]^{\top}$ of the four-wheel steering vehicle under the neural observer.}
        \label{fig: Four-wheel Steering Vehicle uncertainty}
    \end{figure}
    
    Likewise, the corresponding neural observer \eqref{eq: nos for nonintegral-chain linear systems} is designed by  (\romannumeral1) the NNs $\pi_{\theta_i},i=1,2$ are parameterized by three hidden layers ($n_1 = n_2 = n_3 = 3$) with $\operatorname{Leaky\ ReLU}$ as the activation function for all layers; (\romannumeral2) the gain $\epsilon=0.1$; (\romannumeral3) {\color{black}$\widehat{x}_1(0)=\bm{0}_4,\widehat{x}_2(0)=\bm{0}_2$}. The control input is given by the output-feedback control $u=Gy$, where $G$ is designed by the matrix $A+BG$ and is Hurwitz. Hence, it is easy to check the boundedness of the state $x$ and the input $u$. In addition, $(\mathbf{C},\mathbf{A}_{\epsilon})$ is observable, indicating the system complies the whole Assumption \ref{ap: the uncertainty for nonintegral-chain systems}. {\color{black} Furthermore, since $\operatorname{r}(CB_w)=\operatorname{r}(B_w)$, we can apply the unknown input observer (UIO) for comparison with the neural observer \cite{darouach1994uio},} which is described as
    {\color{black} 
    $$
    \left\{\begin{aligned}
    \dot{z}&=N z+L y+G u,\ z(0)=\bm{0}_4, \\
    \hat{x}_{\text{UIO}}&=z-E y,
    \end{aligned}\right.
    $$
    where the matrices $N,L,G,E$ are given by (6)-(12) in \cite{darouach1994uio}. Then, the state $\hat{x}_{\text{UIO}}$ can be the estimate of $x$. We notice that we can simply set the initial value of $\hat{x}_{\text{UIO}}(0)$ by adopting $z(0)=\bm{0}_4$ to reduce the cost of identification of $x(0)$.}
    
    As shown in Fig.~\ref{fig: Four-wheel Steering Vehicle}, the state $x$ can be well-estimated by $\widehat{x}$ with less response time than $\hat{x}_{\text{UIO}}$. Moreover, in Fig.~\ref{fig: Four-wheel Steering Vehicle uncertainty}, the extended states ($x_i,i=5,6$) could be well-estimated by $\widehat{x}_i,i=5,6$ very quickly, which cannot be done by the conventional UIO.

\section{Conclusion and Future work}
    Machine learning meeting control theory is a hot topic worth investigating. In this paper, we creatively introduce the residual neural networks into the design of the observer, called neural observer, and provide the necessary proofs of the convergence.  
    
    More specifically, we propose a new framework to design the neural observers for different dynamical systems, including linear systems and two classes of nonlinear systems with some mild assumptions. The great performance of our proposed observer benefits from the introduction of NNs. Accordingly, we provide specific neural observers for  linear systems, \textcolor{black}{integrator chain} nonlinear systems, and a class of MIMO nonlinear systems composed of a linear dynamic and a general uncertainty. 
    For linear systems, by combining the recent NN controller proposed in \cite{yin2022stability}, we show that the observer could be used in global feedback stabilization. In addition, by using QCs to bound the nonlinear activation functions in NNs, we propose the corresponding LMI conditions for different system settings to achieve neural observability (according to Definition \ref{def: neural observable}). On the other hand, it has also been shown that the observability of system matrices is a necessary condition for the existence of solutions of the aforementioned LMIs. To the best of our knowledge, this is the first time that the neural observability has been discussed theoretically and connected with the observability of a specific system.    

    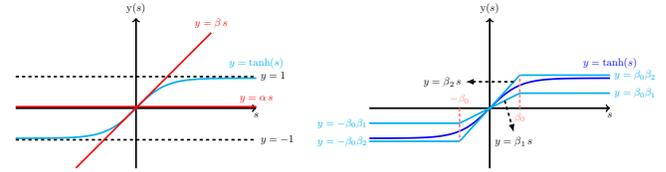
\begin{figure}[t]
      \centering
      \subfigure{\scalebox{0.4}{
           \begin{tikzpicture}
                \centering
                \draw[<->, ultra thick](4,0)node[below] {$s$}--(0,0)--(0,4-1)node[above] {$\operatorname{y}(s)$};
                \draw[ultra thick](-4,0)--(0,0)--(0,-4+2);
                \draw[cyan,domain=-4:4,ultra thick,samples=100] plot(\x,{((e^(1*(\x)))-(e^(-1*(\x))))/((e^(1*(\x)))+(e^(-1*(\x))))}) node at (4,1.5){$y=\tanh(s)$};
                \draw[red,domain=-2:2.5,ultra thick,samples=50] plot(\x,{((1*\x)+0.01)})node[above]{$y=\operatorname{\beta}s$};
                \draw[red,domain=-4:4,ultra thick,samples=50] plot(\x,0.05)node[above]{$y=\operatorname{\alpha}s$};
                \draw[domain=-4:4,ultra thick,samples=50,dashed] plot(\x,1.05)node[right]{$y=1$};
                \draw[domain=-4:4,ultra thick,samples=50,dashed] plot(\x,-1.05)node[right]{$y=-1$};
            \end{tikzpicture} 
        }}
      \subfigure{\scalebox{0.4}{
           \begin{tikzpicture}
                \centering
                \draw[<->, ultra thick](\num,0)node[below] {$s$}--(0,0)--(0,\num-1)node[above] {$\operatorname{y}(s)$};
                \draw[ultra thick](-\num,0)--(0,0)--(0,-\num+2);
                \draw[blue,domain=-4:4,ultra thick,samples=100] plot(\x,{((e^(1*(\x)))-(e^(-1*(\x))))/((e^(1*(\x)))+(e^(-1*(\x))))}) node at (4,1.5){$y=\tanh(s)$};
                
                \draw[cyan,domain=-1:1,ultra thick] plot(\x,{((1.1*\x))});
                \draw[cyan,domain=-1:1,ultra thick] plot(\x,{((0.5*\x))});
                \draw[cyan,domain=1:4,ultra thick] plot(\x,1.1)node[right]{$y=\beta_0\beta_2$};
                \draw[cyan,domain=1:4,ultra thick] plot(\x,0.5)node[right]{$y=\beta_0\beta_1$};
                \draw[cyan,domain=-4:-1,ultra thick] plot(\x,-1.1)node at (-4.9,-1.1){$y=-\beta_0\beta_2$};
                \draw[cyan,domain=-4:-1,ultra thick] plot(\x,-0.5)node at (-4.9,-0.5){$y=-\beta_0\beta_1$};
                \draw[-latex,ultra thick,dashed] (0.8,0.88) -- (-0.8,0.88)node[left]{$y=\operatorname{\beta}_2 s$};
                 \draw[-latex,ultra thick,dashed] (0.5,0.25) -- (0.8,-0.8)node[below]{$y=\operatorname{\beta}_1 s$};
                \draw[red!50,ultra thick,densely dashed] (1,0)node[below]{$\beta_0$}--(1,1.1);
                \draw[red!50,ultra thick,densely dashed] (-1,0)node[above]{$-\beta_0$}--(-1,-1.1);
            \end{tikzpicture} 
        }}
      \caption{\textbf{Left}: Sector constraints on $\tanh(s)$; \textbf{Right}: Piecewise sector constraints on $\tanh(s)$: for $s\in[-\beta_0,\beta_0]$, we have $\left(\tanh(s)-\beta_0 s \right) \left(\beta_1 s-\tanh(s)\right) \geq 0$; for $|s|>\beta_0$, we have $\left(\tanh(s)-\operatorname{sign}(s)\beta_0\beta_2\right) \left(\operatorname{sign}(s)\beta_0\beta_1-\tanh(s)\right) \geq 0$.}
      \label{fig: sector on tanh}
    \end{figure}

    There are some future works that can be done. For instance, we note that the global sector boundedness regarding activation functions introduced in Section \ref{sec: NN Representation and NN mapping vector} is relatively ``strict'', so that some information from activation functions may not be exploited fully. In detail, the left sub-diagram in Fig.~\ref{fig: sector on tanh} shows the global sector using the $\tanh$ function as an example. 
    Although we can describe the activation function $y=\tanh(s)$ roughly by using the open region formed by two straight lines $y=\alpha s, y=\beta s$ passing through the origin, some geometric information about the activation function, such as $\lim_{s\rightarrow +\infty}\tanh(s)=1, \lim_{s\rightarrow -\infty}\tanh(s)=-1, \lim_{s\rightarrow \infty}\frac{\mathrm{d}}{\mathrm{d}s}\tanh(s)=0$, is not fully extracted in Lemmas \ref{lem: QC for one mapping}-\ref{lem: QC for K mappings}.
Since the LMIs \eqref{eq: LMI for LTI without control}, \eqref{eq: LMI for LTI with control}, \eqref{eq: LMI for integral-chain systems} and \eqref{eq: LMI for non-integral chain systems} are all based on Lemmas 1-2, it is obvious that we would ignore some NN architectures, that do not satisfy the LMIs but can still be used in the design of neural observers. 
    Intuitively, we could use the piecewise sectors shown in the right sub-diagram in Fig.~\ref{fig: sector on tanh} to characterize the nonlinear activation functions in the NN, which may make better use of geometric information to improve the results. The remaining question, therefore, arises whether we can find constraint conditions from the piecewise sectors boundedness that can be utilized in neural observers. 
\appendices



\section{Proofs of Propositions}
\subsection{Proposition \ref{prop: ness-cond for extending observable}}
\begin{proof}
    If $A$, $C$, $B_w$ satisfy the extending observable condition, i.e., $(\mathbf{C},\mathbf{A})$ is observable, then for any $s\in \mathbb{C}$, we have
    $$
    \operatorname{r}\left[\begin{array}{cc}
       sI_{n_s}-A & -B_w\\
       O & sI_{n_q}\\
       C& O
    \end{array}\right]=n_s+n_q.
    $$
    We suppose that $(C,A)$ is not observable. Hence, there exists the $s_0\in \mathbb{C}$, such that $\operatorname{r}\left[s_0I_{n_s}-A^{\top},C^{\top}\right]<n_s$. In the case of $s_0\neq0$, we have
    \begin{equation*}
        \begin{aligned}
            \operatorname{r}\left[\begin{array}{cc}
       s_0I_{n_s}-A & -B_w\\
       O & s_0I_{n_q}\\
       C& O
    \end{array}\right]=&\operatorname{r}\left[\begin{array}{cc}
       O & s_0I_{n_q}\\
       s_0I_{n_s}-A & O\\
       C& O
    \end{array}\right]\\
    =&\operatorname{r}\left[s_0I_{n_q}\right]+\operatorname{r}\left[\begin{array}{c}
         s_0I_{n_s}-A  \\
         C 
    \end{array}\right]\\
    <&n_s+n_q.
        \end{aligned}
    \end{equation*}
    And in the case of $s_0=0$, we have
    \begin{equation*}
        \begin{aligned}
            \operatorname{r}\left[\begin{array}{cc}
       s_0I_{n_s}-A & -B_w\\
       O & s_0I_{n_q}\\
       C& O
    \end{array}\right]\leq&\operatorname{r}\left[\begin{array}{c}
         s_0I_{n_s}-A  \\
         C 
    \end{array}\right]+\operatorname{r}\left[B_w\right]\\
    <&n_s+n_q.
        \end{aligned}
    \end{equation*}
    Therefore, the above inequalities lead to a contradiction.
\end{proof}
\subsection{Proposition \ref{prop: existence of solution of LMI}}
\begin{proof}
        First, we unfold and directly compute the matrices in left side in LMI \eqref{eq: LMI for LTI without control} as follows:
        $$
        R_{\pi}^{\top}
        \left[
            \begin{array}{cc}
                A^{\top}P + PA & P\\
                P&O
            \end{array}
        \right]R_{\pi}=
        \left[
            \begin{array}{cc}
                \widetilde{A}^{\top}P+P\widetilde{A} & P N_{\pi w} \\
                \star & O
            \end{array}
        \right],
        $$
        $$
        R_{\xi}^{\top}\Psi_{\sigma}^{\top} M_{\sigma}(\lambda_{\sigma}) \Psi_{\sigma}R_{\xi}=
        \left[
            \begin{array}{cc}
                O & N_{\xi x}^{\top}R_1\\
                \star & \begin{array}{c}
                     N_{\xi w}^{\top}R_1+R_1 N_{\xi w} \\
                     -2\operatorname{diag}(\lambda_{\sigma})
                \end{array}
            \end{array}
        \right].
        $$
        To prove the Proposition \ref{prop: existence of solution of LMI}, we need the following steps.
        
        \textbf{Step 1}: A strictly diagonally dominant diagonal matrix $T = (t_{ij})\in\mathcal{L}(\mathbf{R}^m)$ has positive diagonal entries, which means that for all $i=1,\cdots,m$, we have $|t_{ii}|>\Sigma_{i\neq j}|t_{ij}|$ and $t_{ii}>0$. Then, this matrix is positively definite. Specifically, for all $x\in\mathbf{R}^{m}$,
        $$
        \begin{aligned}
        x^{\top} T x &=\sum_{i=1}^{m} t_{ii} x_{i}^{2}+\sum_{i \neq j} t_{ij} x_{i} x_{j}\\
        &> \sum_{i=1}^{m}\Big(\sum_{i \neq j}\left|t_{ij}\right|\Big) x_{i}^{2}-\sum_{i \neq j}\left|t_{ij}\right|\left|x_{i}\right| \left|x_{j}\right|\\
        &= \sum_{j>i}\Big(\left|t_{ij}\right|\big(x_{i}^{2}+x_{j}^{2}-2\left|x_{i}\right|\left|x_{j}\right|\big)\Big) \geq 0.
        \end{aligned}
        $$
        \textbf{Step 2}: $\bm{\Rightarrow}$ For the sufficiency, due to the observability of $(C,A)$, then $\widetilde{A}=A+N_{\pi x}$ is a Hurwitz matrix by taking the matrix $W^{L+2}$ in $N_{\pi x}=W^{L+2}C$ is a pole assignment matrix for $\widetilde{A}$. Since $\widetilde{A}=A+N_{\pi x}$ is a Hurwitz matrix, we imply that the Lyapunov equation $-(\widetilde{A}^{\top}P+P\widetilde{A})=Q$ has a unique solution $P=\int_0^{\infty}e^{\widetilde{A}^{\top}t}Qe^{\widetilde{A}t}\mathrm{d}t$ that is finite, i.e., $\|P\|_F< \infty$. Therefore, we can rewrite the LMI into
        $$
        \begin{aligned}
        R_{\pi}^{\top}
            \left[
                \begin{array}{cc}
                    A^{\top}P + PA & P\\
                    P&O
                \end{array}
            \right]&R_{\pi}+
            R_{\xi}^{\top}\Psi_{\sigma}^{\top} M_{\sigma}(\lambda_{\sigma}) \Psi_{\sigma}R_{\xi} \\
        \end{aligned}
        $$
        $$
        =-\underbrace{
            \left[
                \begin{array}{cc}
                    Q & M_1 \\
                    M_1^{\top} & 2\operatorname{diag}(\lambda_{\sigma})-M_2
                \end{array}
            \right]}_{\triangleq M_0}.
        $$
        By substituting $N_{\xi w}$ into $M_2$ from above, we can show that $M_2$ is a symmetric matrix with zero diagonal entries. Hence, under the assumptions (\romannumeral1) and (\romannumeral2), the LMI \eqref{eq: LMI for LTI without control} is satisfied due to $M_0$ is strictly diagonally dominant.
        
        $\Leftarrow$ For  the sake of necessity, we assume that $(C,A)$ is unobservable, and there is a matrix $\widetilde{P}$ that makes the LMI \eqref{eq: LMI for LTI without control} accurate. Since $\widetilde{A}$ is not Hurwitz, we imply that all eigenvalues of $-Q=\widetilde{A}^{\top}\widetilde{P} + \widetilde{P}\widetilde{A}$ are not negative, which leads to a contradiction since LMI \eqref{eq: LMI for LTI without control} has no solution. This completes the proof.
    \end{proof}

    \subsection{Proposition \ref{prop: existence of solution of LMI 2}}
\begin{proof}
        \textbf{Sufficiency:} Since $(C, A)$ is observable, and $(A,B)$ is controllable, $\widetilde{A}_1=A+BW_1^{L_1+2}$ and $\widetilde{A}_2=A+W_2^{L_2+2}C$ are two Hurwitz matrices by making $W_1^{L_1+2}$ and $W_2^{L_2+2}$ are pole assignment matrices. Subsequently, it is not difficult to verify that $P_i$ and $Q_i$ satisfy the Lyapunov equation $-(\widetilde{A}^{\top}_iP_i+P_i\widetilde{A}_i)=Q_i,i=1,2$. The matrices on the left side of LMI \eqref{eq: LMI for LTI with control} can be expanded to
        $$
        \widehat{R_{\bm{\pi}}}^{\top}
        \left[
            \begin{array}{cc}
                \widehat{A}^{\top}\widehat{P}  + \widehat{P}\widehat{A} & \widehat{P}\\
                \widehat{P}&O
            \end{array}
        \right]\widehat{R_{\bm{\pi}}}=
        \left[
            \begin{array}{cc}
                -\widehat{Q}+M_3 & \widehat{P}\widehat{N_{\mathbf{\pi} w}}\\
                \star & O
            \end{array}
        \right],
        $$
        $$
        \widehat{R_{\xi}}^{\top}\mathbf{\Psi}(2)^{\top} \mathbf{M}(2) \mathbf{\Psi}(2)\widehat{R_{\xi}}=
        \left[
            \begin{array}{cc}
                O & \widehat{N_{\xi x}}^{\top}R_1\\
                \star & \begin{array}{c}
                    \widehat{N_{\xi w}}^{\top}R_1+R_1 \widehat{N_{\xi w}} \\
                     -2\operatorname{diag}(\lambda^2_{\sigma})
                \end{array}
            \end{array}
        \right].
        $$
        Hence, based on the assumptions (\romannumeral3) and ((\romannumeral4), the LMI \eqref{eq: LMI for LTI with control} is satisfied due to the property of strict diagonal dominance. 

        \textbf{Necessity:} The proof is the same as step 2 in Proposition \ref{prop: existence of solution of LMI}. 
    \end{proof}

\section{Proof of Lemma}

\subsection{Lemma 2}
    \begin{proof}
        The proof is a direct extension of Lemma \ref{lem: QC for one mapping}. Specifically, the left side of above inequality is equivalent to 
        $$
        \sum_{k=1}^K\sum_{i=1}^{n_{\sigma_k}}\lambda_{\sigma_k,i}\underbrace{\left(w_{\sigma_k,i}-\alpha_{\sigma_k,i}\xi_{\sigma_k,i}\right)\left(\beta_{\sigma_k,i}\xi_{\sigma_k,i}-w_{\sigma_k,i}\right)}_{\geq 0,\ \text{due to}\ w_{\sigma_k,i}=\sigma(\xi_{\sigma_k,i})\geq 0.}\geq 0.
        $$
    \end{proof}
\subsection{Lemma \ref{lem: Positive definite under perturbation}}
    \begin{proof}
    The proof of this lemma can be directly provided by 
    $$
    \operatorname{r}\left[\begin{array}{c}
        \mathbf{C}  \\
        \mathbf{C}\mathbf{A}_{\epsilon}\\
        \vdots\\
        \mathbf{C}\mathbf{A}_{\epsilon}^{n_s+n_q-1}
    \end{array}\right]=\operatorname{r}\left[\begin{array}{c}
        \mathbf{C}  \\
        \mathbf{C}\mathbf{A}\\
        \vdots\\
        \mathbf{C}\mathbf{A}^{n_s+n_q-1}
    \end{array}\right]=n_s+n_q.
    $$
    \end{proof}
\section{Proof of Theorem}
\subsection{Theorem \ref{theo: LMI for LTI without control}}
 \begin{proof}
        First, we suppose that the existence of the matrix $P$ is true. Denote $e(t)=\widehat{x}(t)-x(t)$, then from \eqref{eq: linear systems} and \eqref{eq: nos for linear systems}, it is not difficult to obtain
        $$
        \dot{e}(t)=Ae(t) + \pi_{\theta}(Ce(t)).
        $$
        We denote $v(t)\triangleq \pi_{\theta}(Ce(t))$. Equivalently, the form of input of $\pi_{\theta}$ can be regarded as $e(t)$ by updating $W^{L+2}$ and $W^1$ in \eqref{eq: R_pi and R_xi} to $W^{L+2}C$ and $W^1C$, respectively. Recall that $P \succ O$, we define a radially unbounded Lyapunov function $V:\mathbf{R}^{n_s}\rightarrow \mathbf{R},\ e(t)\mapsto e^{\top}(t)Pe(t)$. Then, the time derivative of $V$ along the trajectories of $e(t)$ is given by
        $$
        \begin{aligned}
            \frac{\mathrm{d}V}{\mathrm{d}t}\Big|_{e(t)}=&\dot{e}^{\top}(t)Pe(t)+e^{\top}(t)P\dot{e}(t)\\
            =&(e^{\top}(t)A^{\top}+v^{\top}(t))Pe(t)+e^{\top}(t)P(Ae(t) + v(t))\\
            =&e^{\top}(t)(A^{\top}P+PA)e(t)+2v^{\top}(t)Pe(t)\\
            =&
            [\star]^{\top}
            \left[
                \begin{array}{cc}
                    A^{\top}P + PA & P\\
                    P&O
                \end{array}
            \right]
            \left[\begin{array}{l}
                    e(t) \\
                    v(t)
                \end{array}
            \right],
        \end{aligned} 
        $$
        where ``$\star$'' can be inferred from symmetry. By using the transformation from \eqref{eq: R_pi and R_xi} and the strict LMI \eqref{eq: LMI for LTI without control}, we imply that there exists $\epsilon >0$ such that the left/right multiplication of the LMI by $\left[e^{\top}, w_{\sigma}^{\top}\right]$ and its transpose yields
        $$
        \begin{aligned}
        &{[\star]^{\top}\left[\begin{array}{cc}
            A^{\top}P + PA & P\\
            P&O
        \end{array}\right]\left[\begin{array}{l}
        e(t) \\
        v(t)
        \end{array}\right]} \\
        &+[\star]^{\top} \Psi_{\sigma}^{\top} M_{\sigma}(\lambda_{\sigma}) \Psi_{\sigma}\left[\begin{array}{c}
        {\xi}_{\sigma}(t) \\
        w_{\sigma}(t)
        \end{array}\right] \leq-\epsilon(\left\|e(t)\right\|_2^{2}+\left\|v(t)\right\|_2^{2}).
        \end{aligned}
        $$
        Therefore, by using Gronwall-Bellman inequality \footnote{We consider that $u(t)\in C^1(\mathcal{T}_{\geq 0})$ and Lemma \ref{lem: QC for one mapping}, $c(t)$ and $f(t)$ are continuous functions defined in $t\in\mathcal{T}_{\geq 0}$. If $\dot{u}\leq c(t)u(t)+f(t)$ for all $t\in\mathcal{T}_{\geq 0}$, then we have $u(t) \leq u(0) e^{\int_{0}^{t} c(\rho) d \rho}+\int_{0}^{t} f(s) e^{\int_{s}^{t} c(\rho) d \rho} d s, t \in \mathcal{T}_{\geq 0}$.}, we deduce that $\frac{\mathrm{d}V}{\mathrm{d}t}\Big|_{e(t)}\leq -\epsilon\left\|e(t)\right\|_2^{2}\leq -\frac{\epsilon V(e(t))}{\lambda_{\max}(P)}$, which in turn gives
        $$
        \begin{aligned}
        \lambda_{\min}(P)\|e(t)\|_2^{2} & \leq V(e(t)) \leq e^{-\frac{\epsilon t}{\lambda_{\max }(P)}} V(e(0)) \\
        & \leq \lambda_{\max }(P) e^{-\frac{\epsilon t}{\lambda_{\max }(P)}}\|e(0)\|_2^{2}.
        \end{aligned}
        $$
        As a consequence, we obtain
        $$
        \begin{aligned}
        \|x(t)-\widehat{x}(t)\|_2 \leq& \overbrace{\sqrt{\lambda_{\max }(P) / \lambda_{\min }(P)}}^M e^{-\frac{\epsilon t}{2 \lambda_{\max }(P)}}\|e(0)\|_2\\
        \leq& M e^{-\frac{\epsilon t}{2 \lambda_{\max }(P)}}\{\|x(0)\|_2+\|\widehat{x}(0)\|_2\}.
        \end{aligned}
        $$
        This completes the proof of Theorem \ref{theo: LMI for LTI without control}.
    \end{proof}

\subsection{Theorem \ref{theo: LMI for LTI with control}}

\begin{proof}
       We consider the radially unbounded function $V:\mathbf{R}^{2n_s}\rightarrow \mathbf{R},$ $ \bm{x}(t)\mapsto \bm{x}^{\top}(t)\widehat{P}\bm{x}(t)$ as a candidate Lyapunov function for above system. Therefore, the time derivative of $V$ along the trajectories of \eqref{eq: system and nos for LTI} is given by 
        $$
        \frac{\mathrm{d}V}{\mathrm{d}t}\Big|_{\bm{x}(t)}=
        [\star]^{\top}
            \left[
                \begin{array}{cc}
                    \widehat{A}^{\top}\widehat{P} + \widehat{P}\widehat{A} & \widehat{P}\\
                    \widehat{P}&O
                \end{array}
            \right]
            \left[\begin{array}{l}
                    \bm{x}(t) \\
                    v(\bm{x})
                \end{array}
            \right].
        $$
        Due to the strictness of LMI \eqref{eq: LMI for LTI with control} and Lemma \ref{lem: QC for K mappings}, by left/right multipling the vector $\left[\bm{x}^{\top}, \left(w^2_{\sigma}\right)^\top\right]$ and its transpose, we know that there exists $\epsilon>0$ such that for all $\bm{x}\in\mathbf{R}^{2n_s}$, 
        $$
        \begin{aligned}
            \frac{\mathrm{d}V}{\mathrm{d}t}\Big|_{\bm{x}(t)}\leq&-\underbrace{[\star]^{\top} \mathbf{\Psi}(2)^{\top} \mathbf{M}(2) \mathbf{\Psi}(2)\left[\begin{array}{c}
                {\xi}_{\sigma}^2(t) \\
                w_{\sigma}^2(t)
                \end{array}\right]}_{\geq 0}-\epsilon\|\bm{x}(t)\|^2_2\\
                \leq&-\epsilon\|\bm{x}(t)\|^2_2.
        \end{aligned}
        $$
        Similarly, we have
        $$
        \left\|\bm{x}(t)\right\|_2\leq\sqrt{\frac{\lambda_{\max}(\widehat{P})}{\lambda_{\min}(\widehat{P})}}e^{-\frac{\epsilon t}{2\lambda_{\max}(\widehat{P})}}\left\|\bm{x}(0)\right\|_2.
        $$
        Notice that $\bm{x}^{\top}(t)=[x_1(t)^{\top},x_2(t)^{\top}]$ and that $x_1(t) = x(t)$, and $x_2(t)=\widehat{x}(t)-x(t)$. It is easy to obtain that
        $$
        \|x(t)\|_2+\|x(t)-\widehat{x}(t)\|_2\leq Me^{-\kappa t}\left\{\|x(0)\|_2+\|\widehat{x}(0)\|_2\right\},
        $$
        where $M=2\sqrt{\frac{2\lambda_{\max}(\widehat{P})}{\lambda_{\min}(\widehat{P})}}$, and $\kappa=\frac{\epsilon}{2\lambda_{\max}(\widehat{P})}$. Hence, the system \eqref{eq: linear systems} is {\color{black}neural exponentially observable}, and the state $x(t)$ converges to $0$ as $t \rightarrow \infty$ exponentially, which leads to the completeness of the proof.
    \end{proof}

\subsection{Theorem \ref{theo: LMI for integral-chain systems}}

\begin{proof}
        Firstly, by constructing that $V_0:\mathbf{R}^{n+1}\rightarrow\mathbf{R},\bm{\eta}(t)\mapsto\bm{\eta}(t)^{\top}P\bm{\eta}(t)$, we denote $\lambda_1=\lambda_{\min}(P)$ and $\lambda_2=\lambda_{\max}(P)$ and compute the time derivative of $V_0(\bm{\eta})$ along \eqref{eq: error system for integral-chain systems} as 
        \begin{equation*}
            \begin{aligned}
                \frac{\mathrm{d}V_0}{\mathrm{d}t}\Big|_{\bm{\eta}(t)}
                =& \epsilon^{-1}[\star]^{\top}
                \left[\begin{array}{cc}
                    \widetilde{\mathcal{A}}^{\top}P+P\widetilde{\mathcal{A}} & -P\\
                    -P& O
                \end{array}\right]
                \left[\begin{array}{c}
                    \bm{\eta}(t)\\
                    \bm{\pi_{\theta}}(\bm{\eta})
                \end{array}\right]\\
                &+\frac{\partial V_0}{\partial \eta_{n+1}}h(t),
            \end{aligned}
        \end{equation*}
        Due to the assumption of LMI \eqref{eq: LMI for integral-chain systems} and Lemma \ref{lem: QC for K mappings}, we imply that there exists $\lambda_3>0$ such that
        $$
        \begin{aligned}
            &[\star]^{\top}
                \left[\begin{array}{cc}
                    \widetilde{\mathcal{A}}^{\top}P+P\widetilde{\mathcal{A}} & -P\\
                    -P& O
                \end{array}\right]
                \left[\begin{array}{c}
                    \bm{\eta}(t)\\
                    \bm{\pi_{\theta}}(\bm{\eta})
                \end{array}\right]\leq-\lambda_3\left\|\bm{\eta}(t)\right\|_2^{2} \\
            & -\underbrace{[\star]^{\top} \mathbf{\Psi}(n+1)^{\top} \mathbf{M}(n+1) \mathbf{\Psi}(n+1)
                \left[\begin{array}{c}
                    {\xi}^{n+1}_{\sigma}(t) \\
                    w^{n+1}_{\sigma}(t)
                \end{array}\right]}_{\geq 0}.
        \end{aligned}
        $$ 
        Moreover, we obtain that
        \begin{equation}
            \begin{aligned}
                \frac{\mathrm{d}V_0}{\mathrm{d}t}\Big|_{\bm{\eta}(t)}
                \leq&-\epsilon^{-1}\lambda_3\left\|\bm{\eta}(t)\right\|_2^{2}+\frac{\partial V_0}{\partial \eta_{n+1}}h(t) \\
                \leq&-\epsilon^{-1}\frac{\lambda_3}{\lambda_2}V_0(\bm{\eta})+\frac{\partial V_0}{\partial \eta_{n+1}}h(t),
            \end{aligned}
        \end{equation}
        where $h(t)=\mathcal{F}_t+\sum_{i = 1}^{n-1}x_{i+1}\mathcal{F}_{x_i}+bu\mathcal{F}_{x_n}+\dot{w}\mathcal{F}_w.$
        
        Secondly, by retrieving Assumption \ref{ap: the uncertainty for integral-chain} about the boundedness of $\bm{x}$ and $u(t)$ and continuity of $\chi(\bm{x},w)$, we use the the Heine–Borel theorem to obtain that the uncertain term $h(t)$ is also bounded, i.e., $|h(t)|\leq M_0$
        Hence, by combining the last term with $\mathrm{d}V_0/\mathrm{d}t$, we obtain that
        
        $$
            \frac{\mathrm{d}\sqrt{V_0}}{\mathrm{d}t}\Big|_{\bm{\eta}(t)}\leq-\epsilon^{-1}\frac{\lambda_3}{2\lambda_2}\sqrt{V_0(\bm{\eta})}+\frac{\lambda_2}{\sqrt{\lambda_1}}M_0.
        $$
        Applying the Gronwall-Bellman inequality again implies that
        \begin{equation}
           \sqrt{V_0(\bm{\eta})}\leq\left(\sqrt{V_0(\bm{\eta}(0))}-M_1\right)e^{-\frac{\lambda_3}{2\epsilon\lambda_2}t}+M_1, 
           \label{eq: result of tracking error}
        \end{equation}
        where $M_1=\epsilon\frac{2\lambda_2^2}{\lambda_3\sqrt{\lambda_1}}M_0$. To be specific, if $\epsilon \rightarrow 0^{+}$, we obtain
        $$
        \begin{aligned}
        \sqrt{V_0(\bm{\eta}(0))}e^{-\frac{\lambda_3}{2\epsilon\lambda_2}t}=&\left(\sum_{i=1}^n\left|\frac{(x_i(0)-\widehat{x}_i(0))}{\epsilon^{n+1-i}}\right|^2\right)^{\frac{1}{2}}e^{-\frac{\lambda_3}{2\epsilon\lambda_2}t}\\
        &\rightarrow0^+.\\
        \end{aligned}
        $$
        Therefore, as $\epsilon\rightarrow 0^+$, we obtain that for $t\in\mathcal{T}_{\geq T}$, 
        $$\|\bm{x}-\widehat{\bm{x}}\|_2= \left(\sum_{i=1}^n|\epsilon^{n+1-i}\eta_i|^2\right)^{\frac{1}{2}}\leq\epsilon\|\bm{\eta}\|_2\rightarrow0^+.$$
        Moreover, $|x_{n+1}-\widehat{x}_{n+1}|=|\eta_{n+1}|\leq\|\bm{\eta}\|_2\rightarrow0^+.$
        These complete the whole proof.
    \end{proof}

\subsection{Theorem \ref{theo: LMI for nonintegral-chain systems}}
 \begin{proof}
    Firstly, we denote $x_1(t)\triangleq x(t)$, $x_2(t)\triangleq\mathcal{K}(t)$, then the system \eqref{eq: nonintegral-chain systems} can be rewritten into $$
    \left\{
    \begin{aligned}
    \dot{x}_1=&Ax_1+Bu+B_wx_2,\\
    \dot{x}_2=&\nabla_t \mathcal{K},\\
    y=&\mathbf{C}\left[
    \begin{array}{c}
         x_1  \\
         x_2 
    \end{array}\right].
    \end{aligned}
    \right.
    $$
    
    By denoting the errors $e_i(t)=x_i(t)-\widehat{x} _i(t)$ and $\eta_i(t)=\frac{e_i(\epsilon t)}{\epsilon^{2-i}}$, we have the following formulation:
    $$
    \left[\begin{array}{c}
        \dot{\eta}_1\\
        \dot{\eta}_2
    \end{array}\right]=
    \underbrace{\left[\begin{array}{cc}
        \epsilon A & B_w\\
        O & O
    \end{array}\right]}_{\mathbf{A}_{\epsilon}}
    \underbrace{\left[\begin{array}{c}
        \eta_1\\
        \eta_2
    \end{array}\right]}_{\bm{\eta}(t)}-
    \underbrace{\left[\begin{array}{c}
        \pi_{\theta_1}(\mathbf{C}\bm{\eta})\\
        \pi_{\theta_2}(\mathbf{C}\bm{\eta})
    \end{array}\right]}_{\bm{\pi_{\theta}}(t)}
    +
    \left[\begin{array}{c}
        0\\
        \epsilon \nabla_t \mathcal{K}
    \end{array}\right].
    $$
    In form, $\bm{\eta}(t)$ can be considered as the input of NN mapping vector $\bm{\pi_{\theta}}$. From Assumption \ref{ap: the uncertainty for nonintegral-chain systems}, it is easy to check that $\|\nabla_t \mathcal{K}\|_2\leq M$ with $M>0$. We take the Lyapunov function $V(\bm{\eta})=\bm{\eta}^{\top}\mathbf{P}\bm{\eta}$. Then, by applying Lemma \ref{lem: QC for K mappings} and $D_1\succ O$, we denote $\lambda_1=\lambda_{\min}(\mathbf{P})$ and $\lambda_2=\lambda_{\max}(\mathbf{P})$ and consequently have the following inequality:
    \begin{equation}
        \begin{aligned}
        &\frac{\mathrm{d}V}{\mathrm{d}t}\Big|_{\bm{\eta}(t)}\\
        \textcolor{blue}{=}&[\star]^{\top}
        \left[
            \begin{array}{cc}
                \mathbf{A}_{\epsilon}^{\top}\mathbf{P} + \mathbf{P}\mathbf{A}_{\epsilon}   & -\mathbf{P}\\
                -\mathbf{P}&O
            \end{array}
        \right]
        \left[\begin{array}{l}
                \bm{\eta}(t) \\
                \bm{\pi_{\theta}}(t)
            \end{array}
        \right]+\frac{\partial V}{\partial \eta_2}\epsilon \nabla_t \mathcal{K} \\
        \leq&-[\star]^{\top} \bm{\Psi}(2)^{\top} \mathbf{M}(2) \bm{\Psi}(2)\left[\begin{array}{c}
            {\xi}^2_{\sigma}(t) \\
            w^2_{\sigma}(t)
            \end{array}\right]\\
            &-\kappa\left\|\bm{\eta}(t)\right\|_2^{2}+2\epsilon M\lambda_2\left\|\bm{\eta}(t)\right\|_2\\
        \leq&-\frac{\kappa}{\lambda_2}V(\bm{\eta})+\frac{2\epsilon M\lambda_2}{\sqrt{\lambda_1}}\sqrt{V(\bm{\eta})}.
    \end{aligned}
    \end{equation}
    Uniformly, by using Gronwall-Bellman inequality, we derive $\left\|e_{i}(t)\right\|_2\rightarrow 0^+$ (as $\epsilon\rightarrow 0^+$) from the following inequality:
    $$
    \left\|e_{i}(t)\right\|_2 \leqslant \epsilon^{2-i}\left[\sqrt{\frac{V(\bm{\eta}(0))}{\lambda_1}} e^{-\frac{\kappa t}{2 \lambda_{2} \epsilon}}+\frac{2\epsilon M\lambda_2^2}{\lambda_1 \kappa} (1-e^{-\frac{\kappa t}{2 \lambda_{2}\epsilon}})\right].
    $$
    This completes the proof.
    \end{proof}
    
\bibliographystyle{ieeetr}
\bibliography{nn-observer}  



\begin{IEEEbiography}
{Song Chen} received the bachelor’s degree in mathematics from China University of Petroleum, Beijing, China, in 2020. He is currently working toward the Ph.D. degree in operational research and cybernetics with Zhejiang University, Hangzhou, China.

His research interests include nonlinear control, learning-based control, machine learning theory, and their applications in robotics.
\end{IEEEbiography}

\vspace{-0 mm}
\begin{IEEEbiography}
{Shengze Cai} received the B.Sc. and the Ph.D.
degrees from Zhejiang University, Hangzhou, China,
in 2014 and 2019, respectively. 

He is currently an assistant professor with the College of Control Science \& Engineering, Zhejiang University (ZJU). Prior to joining ZJU in 2022, he was a Post-Doctoral Research Associate
with the Division of Applied Mathematics, Brown
University, Providence, RI, USA. His research interests include scientific  machine learning, data/image processing, control \& optimization as well as  flow visualization techniques. 
\end{IEEEbiography}

\vspace{-0 mm}
\begin{IEEEbiography}
{Tehuan Chen} received the bachelor’s degree from  Hangzhou Dianzi University, Hangzhou, China,  in 2011, and the Ph.D. degree from the College of Control Science and Engineering, Zhejiang University, Hangzhou, in 2016.   

He is currently an Associate Professor with the  School of Mechanical Engineering and Mechanics, Ningbo University, Ningbo, China. His research interests include robotics, optimal control, and distributed parameter systems.
\end{IEEEbiography}
\vspace{-0 mm}
\begin{IEEEbiography}
{Chao Xu} (Senior Member, IEEE), received the Ph.D. degree in mechanical engineering from Lehigh University, Bethlehem, PA, USA, in 2010.

He is currently Associate Dean and Professor of Controls and Autonomous Systems with the College of Control Science $\&$ Engineering, Zhejiang University (ZJU). He serves the inaugural Dean of ZJU Huzhou Institute, as well as plays the role of the Managing Editor for two international journals, e.g., \textit{IET Cyber-Systems and  Robotics} (IET-CSR), and \textit{Journal of Industrial and Management Optimization} (JIMO). His research expertise is Cybernetic Physics and Autonomous Mobility in general, with a focus on, modeling and  control of aerial robotics with applications, machine learning for dynamic systems and control, visual sensing and machine learning for complex fluids.
\end{IEEEbiography}
\vspace{-0 mm}
\begin{IEEEbiography}
{Jian Chu} (Senior Member, IEEE) was born in 1963. He received the B.Sc., M.S., and Ph.D. degrees  from Zhejiang University (ZJU), Hangzhou, China, in 1982, 1984, and 1989, respectively. He attended the joint Ph.D. Program of ZJU and Kyoto University, Kyoto, Japan.

After that, he joined the faculty of ZJU, where he  became a Full Professor in 1993. He is the Founder of the Institute of Cyber-Systems and Control, ZJU. He is also the Founder of the SUPCON Group, Hangzhou, which is considered as the top automation company. His current research interests include industrial process automation and computer control systems (i.e., industrial operating systems and control-module-on-chip).
\end{IEEEbiography}

\end{document}